\documentclass[journal]{IEEEtran}
\usepackage{cite}

\usepackage{moreverb,url,framed}

\usepackage{graphicx} 
\usepackage{float}
\usepackage{subfig}
\usepackage{tabularx}

\usepackage{amssymb}
\usepackage{latexsym}
\usepackage{amsmath}
\usepackage{amsthm}
\usepackage{tensor}
\usepackage{xfrac}

\usepackage{enumerate}

\usepackage{algorithm}
\usepackage{algpseudocode}

\newtheorem{theorem}{Theorem}[section]

\newtheorem{remark}{Remark} [section]

\newtheorem{proposition}{Proposition} [section]

\newtheorem{corollary}{Corollary} [section]
\newtheorem{assumption}{Assumption}
\newtheorem{example}{Example}

\def\R{\mathbb{R}}
\def\F{\mathcal{F}}

\def\P{\mathsf{P}}

\def\ds{\mathrm{d}s}
\def\dX{\mathrm{d}X}
\def\dY{\mathrm{d}Y}
\def\dW{\mathrm{d}W}
\def\Q{\mathsf{Q}}
\def\pR{\mathsf{R}}
\def\bR{{\overleftarrow{\mathsf{R}}}}
\def\fP{{\overrightarrow{\mathsf{P}}}}
\def\fPi{\overrightarrow{\mathsf{P}}_i}

\def\eps{\varepsilon}
\def\argmin{\operatornamewithlimits{arg\,min}}

\def\tr{\operatorname{tr}}
\def\ol{\overline}
\def\wh{\widehat}
\def\wt{\widetilde}
\def\E{\mathbf{E}}

\begin{document}

\title{Solving Feynman-Kac Forward Backward SDEs Using McKean-Markov Branched Sampling} 

\author{Kelsey P. Hawkins, Ali Pakniyat, Evangelos Theodorou, 
    and Panagiotis Tsiotras, 
\thanks{K. P. Hawkins is with the Institute for Robotics and Intelligent Machines,
Georgia Institute of Technology, Atlanta, GA 30332--0250 USA (e-mail: kphawkins@gatech.edu)}
\thanks{A. Pakniyat is with the Department of Mechanical Engineering, University of Alabama, Tuscaloosa, AL 35401 (e-mail:
apakniyat@ua.edu)}
\thanks{E. Theodorou, and P. Tsiotras are with the School of Aerospace Engineering and the
 Institute for Robotics and Intelligent Machines,
Georgia Institute of Technology,  Atlanta, GA 30332--0250 USA (e-mail:
evangelos.theodorou@gatech.edu, tsiotras@gatech.edu)}}


\maketitle

\begin{abstract}
    We propose a new  method for the numerical solution of the forward-backward stochastic differential equations \mbox{(FBSDE)} appearing in the Feynman-Kac representation of the value function in stochastic optimal control problems. 
   Using Girsanov's change of probability measures, it is demonstrated how a McKean-Markov branched sampling method can be utilized for the forward integration pass, as long as the controlled drift term is appropriately compensated in the backward integration pass. 
   Subsequently, a numerical approximation of the value function is proposed by solving a series of function approximation problems backwards in time along the edges of a space-filling tree consisting of trajectory samples.
   Moreover, a local entropy-weighted least squares Monte Carlo (LSMC) method is developed to concentrate function approximation accuracy in regions most likely to be visited by optimally controlled trajectories. The proposed methodology is numerically demonstrated on linear and nonlinear stochastic optimal control problems with non-quadratic running
    costs, which reveal significant convergence improvements over previous FBSDE-based numerical solution methods.
\end{abstract}


\section{Introduction}

\IEEEPARstart{T}{he} Feynman-Kac representation theorem establishes the intrinsic relationship between the solution of a broad class of second-order parabolic and elliptic partial differential equations (PDEs) to the solution of forward-backward stochastic differential equations \mbox{(FBSDEs)} (see, e.g., \cite[Chapter 7]{yong1999stochastic}). 
Investigations over 
these FBSDEs were brought to prominence in \cite{Pardoux1990,peng1993backward,el1997backward}, and they
have been gaining traction as a framework to solve stochastic nonlinear control problems, including optimal control problems with quadratic cost \cite{exarchos2018stochastic}, minimum-fuel ($L_1$-running cost) problems \cite{Exarchos2018}, differential games \cite{Exarchos2016, Exarchos2018a}, and reachability problems \cite{exarchos2018stochastic,mete2002stochastic}. 
FBSDE-based numerical methods have also received interest from the mathematical finance community \cite{Bender2007,Longstaff2001,ma2007forward}. 
This is due to that fact that the Hamilton-Jacobi-Bellman (HJB) second order PDE appearing in stochastic optimal control can be solved via FBSDE methods with general nonlinear dynamics and~costs. 
Although initial results demonstrate promise in terms of flexibility and theoretical validity, numerical algorithms which leverage this theory have not yet matured. 
For even modest problems, state-of-the-art algorithms often have issues with slow and unstable convergence to the optimal policy. 
Producing more robust numerical methods is critical for the broader adoption of FBSDE methods for real-world tasks.


FBSDE numerical solution methods broadly consist of two steps, a forward pass, which generates Monte Carlo samples of the forward stochastic process, and a backward pass, which iteratively approximates the value function backwards in time. 
Typically, FBSDE methods perform this approximation using a least-squares Monte Carlo (LSMC) scheme, which implicitly solves the backward SDE using parametric function approximation \cite{Longstaff2001}. 
The approximate value function fit in the backward pass is then used to improve sampling in an updated forward pass, leading to an iterative algorithm which, ideally, improves the approximation, till convergence. 
Although FBSDE methods share a distinct similarity to differential dynamic programming (DDP) techniques \cite{jacobson1970differential,theodorou2010stochastic,NIPS2007_3297}, 
as they also involve forward and backward passes,
the latter are, in general, less flexible. 
For most DDP applications a strictly positive definite running cost with respect to the control is required for convergence \cite[Section 2.2.3]{tassa2011theory}. 
Furthermore, in DDP the computation of first and second order derivatives of both the dynamics and the cost is necessary for the backward pass, making it challenging to apply this approach to problems where these derivatives cannot be computed analytically.
In contrast, FBSDE techniques only require a good fit of the value function and the evaluation of its gradient.


A key feature of FBSDE methods is their ability to generate a parametric model for the value function over the entire time horizon which, in turn, can be used for the evaluation and assessment of the  performance of closed-loop  control policies. 
This feature differentiates both FBSDE and DDP methods from model predictive control (MPC) methods~\cite{williams2017information}, which, in general, only produce the current-best optimal control signal, re-evaluated at every time step~\cite{garcia1989model}.

FBSDE methods provide an alternative to grid-based methods for solving PDEs, typically
utilizing
finite-difference, finite-element, or level-set schemes, which are known to scale poorly in high dimensional state spaces ($n \geq 4$). 
There is also ample research into the development of meshless methods for solving
    PDEs, such as radial basis function (RBF) collocation and RBF-finite difference (RBF-FD) 
    formulations \cite{Yensiri2017}.
FBSDE methods share significant similarities with these methods, in the sense that the
    value function is approximated by solving the PDE over an unstructured set
    of collocation points.
The primary drawback of RBF methods is that they do not offer an efficient method
    for choosing the collocation points, and since
    it is difficult to know a priori what the best points are, point selection might
    regress into a grid-based method.
Specifically, sufficiently broad and dense sampling of a high-dimensional state space might require
    roughly the same number of collocation points as a grid-based method to be well-conditioned and
   to produce a quality estimate of the value function \cite{Chen2013}.

While Feynman-Kac-based FBSDE methods 
produce an unbiased estimator for the value function associated with HJB~equations, 
a na\"\i ve application of the theory leads to estimators with high variance by producing sample trajectories away from the optimal one.
Recent work has shown that Girsanov's change of probability measures can be employed to make changes to the sampling in the forward pass without adding intrinsic bias to the estimator \cite{exarchos2018stochastic,Exarchos2018,Exarchos2018a}. 
In other words, the drift appearing in the forward SDE as a consequence of the change of the probability measure can be employed to modify the sampling in the forward pass; this, in turn, requires appropriate accommodation for the change of measure and the associated conditional expectations in the backward pass.

In this work we expand upon the above ideas and invoke Girsanov's theorem for Feynman-Kac FBSDEs in a broader setting than that of \cite{exarchos2018stochastic,Exarchos2018,Exarchos2018a} and show that the forward sampling measure can be modified at will; this enables us to incorporate methods from other domains, namely, rapidly-exploring random trees (RRTs)  (see, e.g., \cite{lavalle2001randomized} and the recent survey in \cite{noreen2016optimal}), in order to more efficiently explore the state space during the forward pass.
RRTs are frequently applied to reachability-type motion planning problems, by biasing the samples towards regions of the state space that have low density. 
Using RRTs in the forward sampling allows us to spread samples evenly over the reachable state space, increasing the likelihood that near-optimal samples are well-represented in the forward pass sample distribution. By sampling more efficiently and relying less on incremental approximations of the value function to guide our search, we can achieve faster and more robust convergence than previous FBSDE methods.
In the backward pass, we take advantage of the path-integrated running costs and the estimates of
    the value function to produce a heuristic which weighs paths 
    according to a local-entropy measure-theoretic optimization.
Although local-entropy path integral theory and RRTs have been used together in \cite{arslan2014}, the method of this article is more closely related to the path-integral approach to control \cite{theodorou2010stochastic}.
Our method, similarly, performs forward passes to broadly sample the state space, but, in contrast to~\cite{arslan2014}, it follows each forward pass with a backward pass to obtain an approximation of the value function, and, consequently, obtain a closed-loop policy over the full horizon.

The primary contributions of this paper are as follows:
\begin{itemize}
    \item Providing the theoretical basis for the use of McKean-Markov branched sampling in the forward pass of FBSDE techniques.
    \item Introducing an RRT-inspired algorithm for sampling the forward SDE.
    \item Presenting a technique for concentrating value function approximation accuracy in regions containing optimal trajectories.
    \item Proposing an iterative numerical method for the purpose of approximating the optimal value function and its policy.
\end{itemize}

This paper expands upon the authors' prior work in \cite{CDC2021Submission}, by: 
first, providing missing proofs and adding the details for proving  all for stated theorems;
second, providing a comprehensive discussion of the proposed algorithm; and
third, by providing additional examples that further illustrate the theory 
and motivate the design choices of the proposed work.


The structure of the paper is as follows. 
Section~\ref{sec:HJB} presents the stochastic optimal control (SOC) problem formulation,
    the on-policy value function representation, and the associated 
    family of Hamilton-Jacobi equations. 
Next, Section~\ref{sec:fkgrep} introduces a constructive series of alternative representations of 
    the value function,
    first as the solution of ``on-policy'' FBSDEs which arise from the Feynman-Kac theorem,
    then as the solution of ``off-policy'' FBSDEs which arise from the application of Girsanov's theorem,
    and finally as the minimizer of a local-entropy weighted optimization problem over the off-policy
    FBSDE distribution.
Section~\ref{sec:branchpathlsmc} discusses branched forward SDE sampling and its novel interpretation
    as a discrete approximation of the continuous-time theory of Section~\ref{sec:fkgrep}.
In Section~\ref{sec:fbrrt} we propose the method of forward-backward rapidly exploring random trees 
    for solving SDEs (FBRRT-SDE),
    a particular implementation of the representation introduced in Section~\ref{sec:branchpathlsmc}. 
Finally, in Section~\ref{sec:numresults}
    we apply FBRRT-SDE to three problems
    and demonstrate its ability to solve nonlinear stochastic optimal control problems 
    with non-quadratic running costs.

\section{Hamilton-Jacobi Equation and \\On-Policy Value Function} 
\label{sec:HJB}

In this section, we briefly introduce the stochastic optimal control (SOC) problem
    under consideration and its associated optimal value function,
    as well as the on-policy value function
    and its associated Hamilton-Jacobi PDE.
Given an initial time $t \in [0,T]$ and a complete filtered probability space
    ${(\Omega, \F, \{\F_s\}_{s \in [t,T]}, \P)}$ on which  the $n$-dimensional standard
    Brownian (Wiener) process $W_s$ is defined,
    consider a stochastic system with dynamics governed by
\begin{align}
    \dX_s &= f(s,X_s,u_s)\, \ds + \sigma(s,X_s) \, \dW_s \text{,} & X_t &= x_t \text{,}
    \label{eq:SOCdyn}
\end{align}
over the interval $[t,T]$,
where $X_s$ is an $\R^n$-valued, progressively measurable state process on the interval $s \in [t,T]$,
$u_{[t,T]}$ is a progressively measurable input process on the same interval
taking values in the compact set $U \subseteq \R^m$, and
$f: [0,T] \times \R^n \times U \rightarrow \R^n$, $\sigma: [0,T] \times \R^n \rightarrow \R^{n \times n}$ 
are the Markovian drift and diffusion functions, respectively.

For each $t \in [0,T]$, the cost over the time interval $[t,T]$ 
associated with a given control signal $u_{[t,T]}$ is
\begin{align}
    S_t[x_t,u_{[t,T]}] &:= \int^T_t \ell(s,X_s,u_s) \, \ds + g(X_T) \text{,} 
\end{align}
where $\ell: [0,T] \times \R^n \times U \rightarrow \R_+$ 
is the running cost, and \mbox{$g: \R^n \rightarrow \R_+$} is the terminal cost.
The stochastic optimal control (SOC) problem is to determine, given $(f, \sigma, \ell, g, T)$,
    the optimal value function $V^*: [0,T] \times \R^n \rightarrow \R_+$, defined as
\begin{align}
    V^*(t,x) &= \inf_{u_{[t,T]} \in \mathcal{U}_{[t,T]}}  \E \big [\, S_t[x,u_{[t,T]} ] \,\big ] 
    \tag{SOC} \label{eq:SOC} ,
\end{align}
    where $\mathcal{U}_{[t,T]}$ is a set of admissible control processes satisfying the conditions
    described in the beginning of \cite[Chapter~4, Section~3]{yong1999stochastic}.
Among these conditions, $u_{[t,T]}$ must be progressively measurable,
the SDE
    \eqref{eq:SOCdyn} must admit a unique solution $X_s$ under the control $u_{[t,T]}$, and $f(\cdot, X(\cdot), u(\cdot))$
    must be $L^1$ integrable.
    
Consider the HJB PDE \cite[Chapter~7, (4.18)]{yong1999stochastic}
\begin{align}
    \begin{aligned}
        \partial_t V^* + \frac{1}{2} \tr[\sigma \sigma^\top \partial_{xx} V^*] 
        &+ h^* (t,x,\partial_{x} V^*) = 0 \text{,} \nonumber \\
        V^*(T,x) = &g(x) ,
    \end{aligned} \tag{HJB} \label{eq:hjbpde}
\end{align}
    where,
\begin{align}
    h^*(t,x,p) &:= \min_{u \in U} \{ \ell(t,x,u) + p^\top f(t,x,u)  \} \text{,} 
    \label{eq:hamil}
\end{align}
and where
    $\partial_t V^*$ is the partial derivative with respect to $t$,
    $\partial_x V^*$ is the gradient with respect to state~$x$,
    and $\partial_{xx} V^*$ is the Hessian with respect to state~$x$.
%

\begin{assumption} \label{ass:optviscosity} 
We assume that the following conditions hold for the problem data.

\begin{enumerate}[i)]
    \item $f, \sigma, \ell, g$ are uniformly continuous, Lipschitz in $x$,
    for all $(t, u) \in [0,T] \times U$.
    \item $\sigma$ has sublinear growth in $x$ and $g$ has polynomial growth in $x$.
    \item $\sigma^{-1}$ exists and is bounded. 
\end{enumerate}
\end{assumption}

Under Assumption~\ref{ass:optviscosity},
    there exists a unique viscosity solution $V^*$ to \eqref{eq:hjbpde},
    there exists a unique weak solution $X_s$ to \eqref{eq:SOCdyn} under each admissible control
    process $u_s$,
    and the viscosity solution $V^*$ has the representation \eqref{eq:SOC}
    \cite[Chapter~7, Theorem~4.4]{yong1999stochastic}.
   
In addition to the optimal value function $V^*$, we also seek to find an optimal feedback control policy
    $\pi^*$.
    According to
    \cite[Chapter~5, Definition~6.1]{yong1999stochastic},
    we define an admissible feedback control policy $\pi:[0,T] \times \R^n \rightarrow U$ as a  measurable
    function for which the SDE
\begin{align}
    \dX_s &= f(s,X_s,\pi(s,X_s)) \, \ds + \sigma(s,X_s) \, \dW_s, \quad X_t = x \text{,}
    \label{eq:fsdeorig} 
\end{align}
has a weak solution.
Due to the boundedness of $\sigma^{-1}$, 
    there exists an optimal feedback control policy $\pi^*$, which satisfies
\begin{align}
\hspace*{-3mm}
    \pi^*(s,x) \in \argmin_{u \in U} \{ \ell(s,x,u) + (\partial_x V^*(s,x))^\top f(s,x,u) \} \text{,}
    \label{eq:optpoli}
\end{align}
    with the verification property that $V^*(t,x) = \E[\, S_t[x, \pi^*] \,]$
    \cite[Chapter~5, Theorem~6.6]{yong1999stochastic}.
    



In this paper, instead of a direct solution of \eqref{eq:hjbpde}, we work with a class of arbitrary control policies $\mu$ and their associated ``on-policy'' value functions $V^\mu$, and we use iterative methods to approximate $V^*$ and $\pi^*$. 
We allow more general policies of the form $\mu : [0,T] \times \R^n \times \R^n \rightarrow U$, that is,
$\mu(t,x,z)$ is now a function of not only the state $x$ but also on an auxiliary variable $z$.
    
The Hamilton-Jacobi (HJ) PDE for the on-policy value function corresponding to the policy
$\mu$ is given by
\begin{gather}
    \partial_t V^{\mu} + \frac{1}{2} \tr[\sigma \sigma^\top \partial_{xx} V^{\mu}] 
    + h^\mu(t,x,\sigma^\top \partial_{x} V^{\mu})  = 0 \text{,} \nonumber \\
    V^{\mu}(T,x) = g(x) \text{,} \tag{HJ} \label{eq:hjpde}
\end{gather}
where,
\begin{equation}  \label{eq:hHJ}
\begin{split}
    &h^\mu(t,x,z) \\ 
    &\quad := \ell(t,x,\mu(t,x,z)) + z^\top \sigma^{-1}(t,x) f(t,x,\mu(t,x,z)) .
    \end{split}
\end{equation}


\begin{assumption} \label{ass:onpolicy_easy} 
We will assume that $\mu$ is chosen such that the function $h^\mu$ in \eqref{eq:hHJ} satisfies the following properties:
\begin{enumerate}[i)]
    \item $h^\mu$ is uniformly continuous in $(t,x)$, Lipschitz in $z$.
    \item $h^\mu$ has polynomial growth in $x$.
\end{enumerate}
\end{assumption}

Under Assumptions~\ref{ass:optviscosity} and Assumption~\ref{ass:onpolicy_easy},
\eqref{eq:hjpde} admits a unique viscosity solution
    \cite[Chapter~7, Theorem~4.5]{yong1999stochastic}.
However, for the ease of presentation henceforth we will impose additional regularity assumptions for the solutions of \eqref{eq:hjpde}.

\begin{assumption} \label{ass:onpolicy_strong} 
The PDE \eqref{eq:hjpde} admits a classical solution, that is, $V^\mu$ is continuously differentiable in $t$, and twice continuously differentiable in~$x$.

\end{assumption}

\begin{remark}
Assumption~\ref{ass:onpolicy_strong} is rather strong for many optimal control problems.
For a set of sufficient conditions for Assumption~\ref{ass:onpolicy_strong} to hold, one might refer, e.g., to 
\cite[Chapter~7, Theorem 5.5]{yong1999stochastic} and \cite[Chapter 4, Theorem 4.2, Theorem 4.4]{fleming2006controlled}.
However, it is well-known  that viscosity solutions can be approximated closely by classical solutions (see, for example, \cite[Section~5.2]{yong1999stochastic}).
Subsequently, our approach provides a numerical approximation of a wide class of optimal control problems, even ones that admit a unique viscosity solution.
The numerical examples in Section~\ref{sec:numresults} demonstrate this claim.
\end{remark}

In our approach, 
we use the continuous function $\partial_x V^\mu$ to characterize the policy's auxilary variable $z$,  henceforth 
letting $\mu(t,x) := \mu(t,x,\sigma^\top \partial_x V^\mu(t,x))$.
For notational conciseness we will further denote $f^\mu(t,x) := f(t,x,\mu(t,x))$ and $\ell^\mu(t,x) := \ell(t,x,\mu(t,x))$.
As it will be shown in Section~\ref{sec:interpretation}, under Assumptions~\ref{ass:optviscosity}-\ref{ass:onpolicy_strong},
    the classical solution $V^\mu$ of \eqref{eq:hjpde} has the interpretation
\begin{align}
    V^\mu(t,x) &= \E \bigg [ \int^T_t \ell^\mu(s,X_s) \, \ds + g(X_T) \bigg ] \text{,}
\end{align}
    where $X_s$ satisfies the SDE
\begin{align}
    \dX_s &= f^\mu(s,X_s) \, \ds + \sigma(s,X_s) \, \dW_s, \quad X_t = x \text{.}
\end{align}
Also, in Section~\ref{sec:interpretation} we revisit the case of non-classical viscosity solutions, demonstrating that this
    interpretation is still valid under certain modifications.

\section{Feynman-Kac-Girsanov FBSDE Representation} \label{sec:fkgrep}
    
\subsection{Off-Policy Drifted FBSDEs} \label{sec:OffPolicy}

In language originating from reinforcement learning,
    an ``on-policy'' method learns a value function $V^\mu$ from trajectory samples
    generated by following
    the same policy $\mu$, whereas ``off-policy'' methods learn from trajectory samples generated by following
    a different policy \cite{sutton2018reinforcement}.
Off-policy methods are generally more desirable because disentangling the sampling distribution from the target
    value function being learned allows for broader exploration and thus more rapid convergence to the
    optimal value function.
Following this language, in stochastic control an ``on-policy'' method would sample from the FSDE \eqref{eq:fsdeorig}
    with drift $f^\mu$ to learn $V^\mu$, but an ``off-policy'' method would sample from a different FSDE to learn $V^\mu$.
In this section we present an off-policy stochastic control method for representing $V^\mu$.
The on-policy version of the theory then arises naturally from a particular specialization.

The off-policy method utilizes the connection between the solution of a pair of FBSDEs
    and the on-policy value function $V^\mu$ solving \eqref{eq:hjpde}.
We first introduce a class of off-policy FBSDEs, and then provide a theorem establishing their
    connection to $V^\mu$.
We call this class of FBSDEs \textit{off-policy} because the drift term of the FSDE
    is a random process that can be chosen at will.

As before, let ${(\Omega, \F, \{\F_t\}_{t \in [0,T]}, \P)}$ be a complete, filtered probability space,
    and let $W^\P_s$ be a Brownian process in the measure $\P$.
 Denote $f^\mu_s := f^\mu(s,X_s)$, and similarly
    for $\ell^\mu_s, \sigma_s$, and let
    $\E_\P$ refer to the
    expectation taken in the probability measure $\P$.
Further, let $K_s$ be an arbitrary $\F_s$-progressively measurable process such that 
    $\sigma^{-1}_s K_s$
    satisfies
Novikov's criterion
($\E_\P[\exp(\sfrac{1}{2} \int_0^T \|\sigma^{-1}_s K_s\|^2 \, \ds)] < \infty$)\cite[Lemma~9]{Lowther2010}.
We call the pair of FBSDEs
\begin{align}
    \dX_s &= K_s \, \ds + \sigma_s \, \dW^\P_s, & X_t &= x \text{,}
    \label{eq:driftfsde} \\
    \dY_s &= -(\ell^\mu_s + Z^{\top}_s D_s) \, \ds + Z^{\top}_s \dW_s^{\P}, & Y_T &= g(X_T)  \text{,}
    \label{eq:driftbsde}
\end{align}
    where,
\begin{align}
    D_s &:= \sigma_s^{-1} (f^\mu_s - K_s) \text{,} \label{eq:dsprocess}
\end{align} 
    the off-policy \textit{drifted FBSDEs} for the target policy $\mu$ and drift process $K_s$.
A solution to (\ref{eq:driftfsde})-(\ref{eq:driftbsde}) is the triple of $\F_s$-adapted processes $(X_s,Y_s,Z_s)$
    for which $X_s$ satisfies the FSDE \eqref{eq:driftfsde}
    and $Y_s,Z_s$ satisfy the BSDE 
    \eqref{eq:driftbsde}.
    
\begin{theorem} \label{thm:driftfbsde}
Let Assumptions~\ref{ass:optviscosity}-\ref{ass:onpolicy_strong} hold, and let $V^\mu$ be a classical solution
    of \eqref{eq:hjpde}.
Assume $K_s$ is chosen such that
    \eqref{eq:driftfsde} admits a unique square-integrable solution $X_s$
    (i.e., it satisfies the properties of \cite[Chapter~1, Theorem~6.16]{yong1999stochastic})
    and $\sigma^{-1}_s K_s$ satisfies Novikov's criterion.
Then, there exists a solution $(Y_s, Z_s)$ to the BSDE \eqref{eq:driftfsde},
    and it holds that
\begin{align}
    Y_s &= V^\mu(s,X_s) \text{,} & s \in [t,T] \text{,} \label{eq:yv} \\
    Z_s &= \sigma_s^\top \partial_x V^\mu(s,X_s) \text{,} & \text{a.e.} \;  s \in [t,T] \text{,} \label{eq:zv}
\end{align}
    $\P$-a.s., and, in particular,
\begin{align}
    Y_t &= \E_{\P}[\wh{Y}_{t,\tau}] = V^\mu(t,x), 
    & \P\text{-a.s.} \text{,}
    \label{eq:thmcontinfk2}
\end{align}
    where,
\begin{align}
    \wh{Y}_{t,\tau} &:= Y_\tau + \int_t^\tau (\ell^\mu_s + Z^{\top}_s D_s) \, \ds \text{.} 
    \label{eq:bsdediff2}
\end{align}
\hfill $\square$
\end{theorem}

\begin{proof}
Since \eqref{eq:driftfsde} has a square-integrable solution,  
    there exists a solution $(Y_s,Z_s)$ to \eqref{eq:driftbsde} which is unique and square-integrable
    because $\ell^\mu_s + z^\top D_s$ is Lipschitz in $y,z$
    \cite[Chapter~7, Theorem~3.2]{yong1999stochastic}.
Girsanov's theorem (see, e.g. \cite[Theorem~10]{Lowther2010}) indicates that if
    we construct the measure $\Q$ from $\P$ via the Radon-Nikodym derivative
\begin{align}
    \mathrm{d} \Q &= \Theta^{\Q|\P}_{t,T} \mathrm{d} \P \text{,}
\end{align}
    where,
\begin{align}
    \Theta^{\Q|\P}_{t,t'} &:= \exp \bigg ( -\frac{1}{2} \int_t^{t'} \| \sigma^{-1}_s K_s \|^2 \, \mathrm{d} s
    - \int_t^{t'} (\sigma^{-1}_s K_s)^\top \mathrm{d}W^\P_s  \bigg ) \text{,} \label{eq:thetadef2}
\end{align}
    then the process
\begin{align}
    W_{t'}^\Q &:= W_{t'}^\P + \int_t^{t'} \sigma^{-1}_s K_s \, \mathrm{d} s \text{,} \label{eq:wqwp}
\end{align}
    is Brownian in the newly constructed measure $\Q$.
A further consequence of Girsanov's theorem (see e.g., \cite[Chapter 5, Theorem 10.1]{fleming1976deterministic})
    is illustrated, through an abuse of notation, by
    substituting the relationship \mbox{$\mathrm{d} W^\P_s = \mathrm{d} W^\Q_s - \sigma^{-1}_s K_s \, \mathrm{d} s$}
    into the equations (\ref{eq:driftfsde}) and (\ref{eq:driftbsde}). 
Performing this substitution yields that the solution $(X_s, Y_s, Z_s)$ to 
    (\ref{eq:driftfsde}-\ref{eq:driftbsde}) with the $\P$-Brownian process $W^\P_s$ also solves the zero-drift FBSDE
\begin{align}
    \dX_s &=   \sigma_s \, \dW_s^{\Q}, \qquad X_t = x \text{,} \label{eq:fsdeorig2}  \\
    \dY_s &= -(\ell^\mu_s + Z^{\top}_s \sigma^{-1}_s f^\mu_s) \, \ds + Z^{\top}_s \, \dW_s^{\Q}, \quad Y_T &= g(X_T)  \text{,} 
    \label{eq:bsdeorig2}
\end{align}
    with $\Q$-Brownian $W^\Q_s$.

Under Assumptions~\ref{ass:optviscosity}-\ref{ass:onpolicy_easy},
    \cite[Chapter~7, Theorem~4.5]{yong1999stochastic}
    it follows that there exists a unique solution $(\wt{X}_s, \wt{Y}_s, \wt{Z}_s)$ to the FBSDE 
\begin{align}
    \mathrm{d} \wt{X}_s &= \sigma_s \, \dW_s^{\Q}, & \wt{X}_t &= x \text{,} \label{eq:fsde_alternate} \\
    \mathrm{d} \wt{Y}_s &= -h^\mu(s,\wt{X}_s,\wt{Z}_s) \, \ds + \wt{Z}^{\top}_s \, \dW_s^{\Q},
    & \wt{Y}_T &= g(\wt{X}_T)  \text{.} \label{eq:bsde_alternate}
\end{align}
    and \mbox{$\wt{Y}_s = V^\mu(s,\wt{X}_s)$} holds $\Q$-a.s., and, in particular, \mbox{$\wt{Y}_t = V^\mu(t,x)$}.
Under Assumption~\ref{ass:onpolicy_strong}, that $V^\mu$ is a classical solution, we also have
    that $\wt{Z}_s = \sigma^\top_s \partial_x V^\mu(s,\wt{X}_s)$ holds $\Q$-a.s.
    \cite[Chapter~7, (4.29)]{yong1999stochastic}.
It follows that 
\begin{align}
    h^\mu(s,\wt{X}_s,\wt{Z}_s) = \ell^\mu_s + \wt{Z}^{\top}_s \sigma^{-1}_s f^\mu_s \text{.}
    \label{eq:hmuequals}
\end{align}
Thus, by the uniqueness of solutions \cite[Chapter~1, Definition~6.6; Chapter~7, Definition~2.1]{yong1999stochastic},
    it follows that $(X_s, Y_s, Z_s)$ also solves 
    (\ref{eq:fsde_alternate})-(\ref{eq:bsde_alternate}) and thus \mbox{$Y_s = V^\mu(s,X_s)$},
    \mbox{$Z_s = \sigma^\top_s \partial_x V^\mu(s,X_s)$} hold $\Q$-a.s..

A further consequence of Girsanov's theorem is that $\P$ and $\Q$ are equivalent measures \cite{Lowther2010},
    that is, $\P(N)=0$ iff $\Q(N)=0$ for $N \in \F$.
Thus the previous statements said to hold $\Q$-a.s. also hold $\P$-a.s..

To show \eqref{eq:thmcontinfk2}, note that 
from \eqref{eq:bsdediff2} and the definition of the It\^o integral \cite[p.~33, (5.23)]{yong1999stochastic}, we have
\begin{align*}
    \wh{Y}_{t,\tau} &= Y_t 
        - \int_t^\tau Z^{\top}_s \dW_s^{\Q}  \text{.} 
\end{align*}
Taking the conditional expectation of both sides, and by noting that
by a basic property of the It\^o integral \cite[p.~34, (5.26)]{yong1999stochastic},
\begin{align*}
    \E_{\P}\bigg[\int_t^\tau Z^{\top}_s \dW_s^{\P} \bigg] &= 0 \text{,} 
\end{align*}
    we have $\E_{\P}[\wh{Y}_{t,\tau}] = \E_{\P}[Y_t] =  \E_{\P}[V^\mu(t,x)] = V^\mu(t,x)$.
\end{proof}

\begin{figure}
\centering
\subfloat[$K_s = X_s$ ($= f^\pi_s$ with $\pi \equiv 0$)]{\label{fig:nodriftexamp}\includegraphics[width=0.95\linewidth]{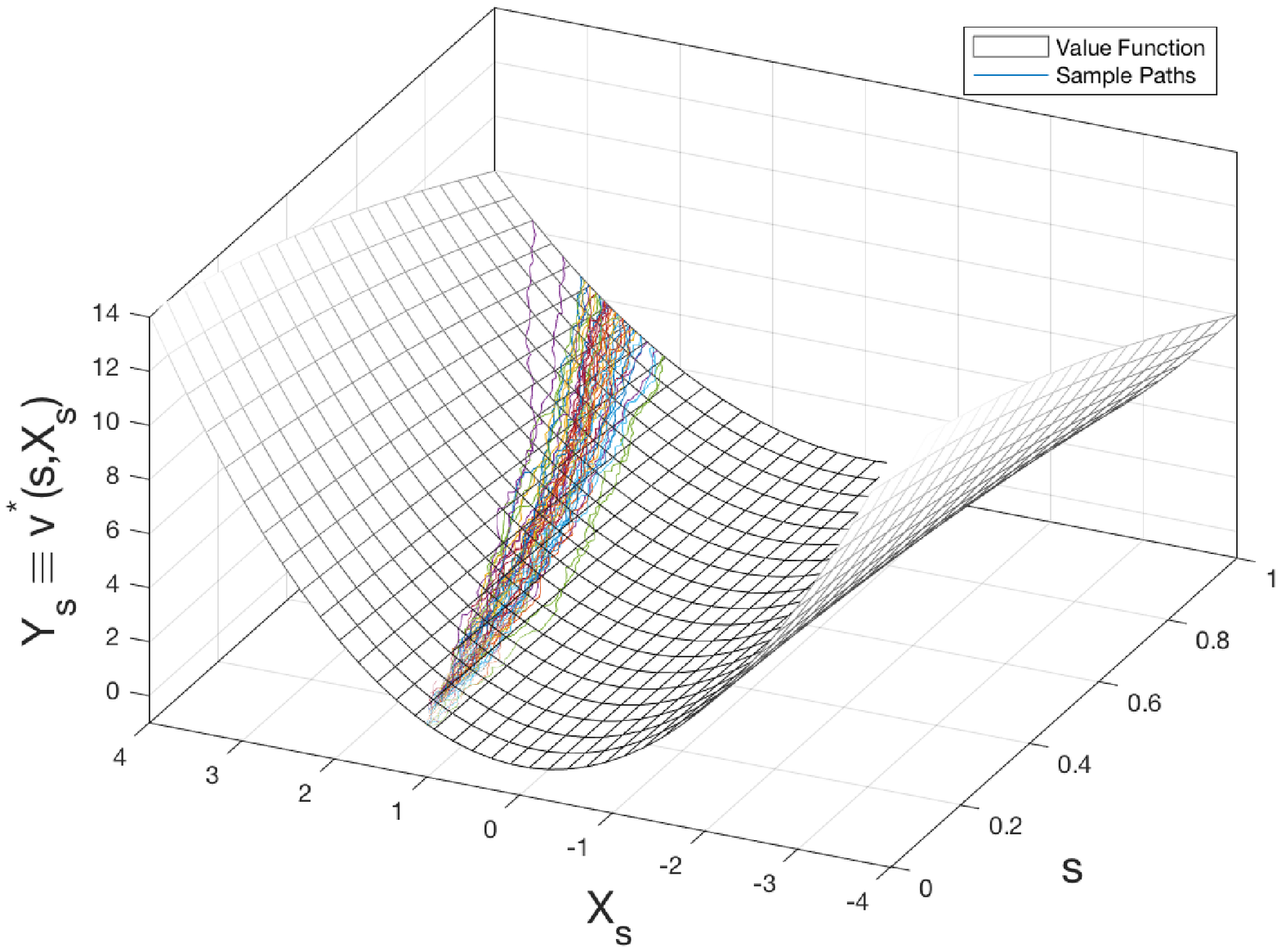}}
\\
\subfloat[$K_s = \sfrac{3}{5} X_s$ ($= f^\pi_s$ with $\pi = -\sfrac{2}{5} X_s$)]{\label{fig:arbitdriftexamp} \includegraphics[width=0.95\linewidth]{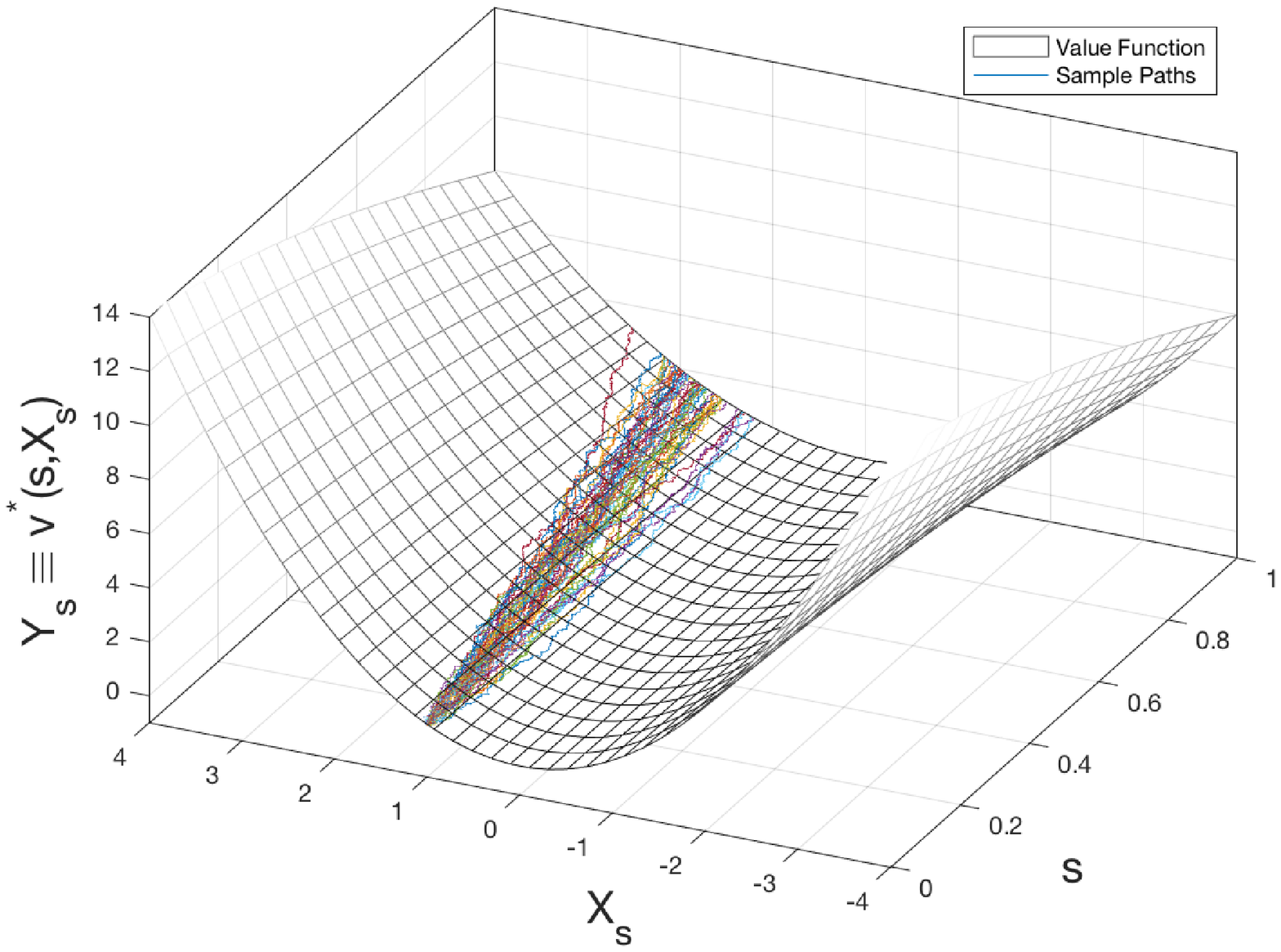}}
\\
\subfloat[$K_s = f^{\pi^*}_s$ (Optimal policy guided dynamics)]{\label{fig:optdriftexamp} \includegraphics[width=0.95\linewidth]{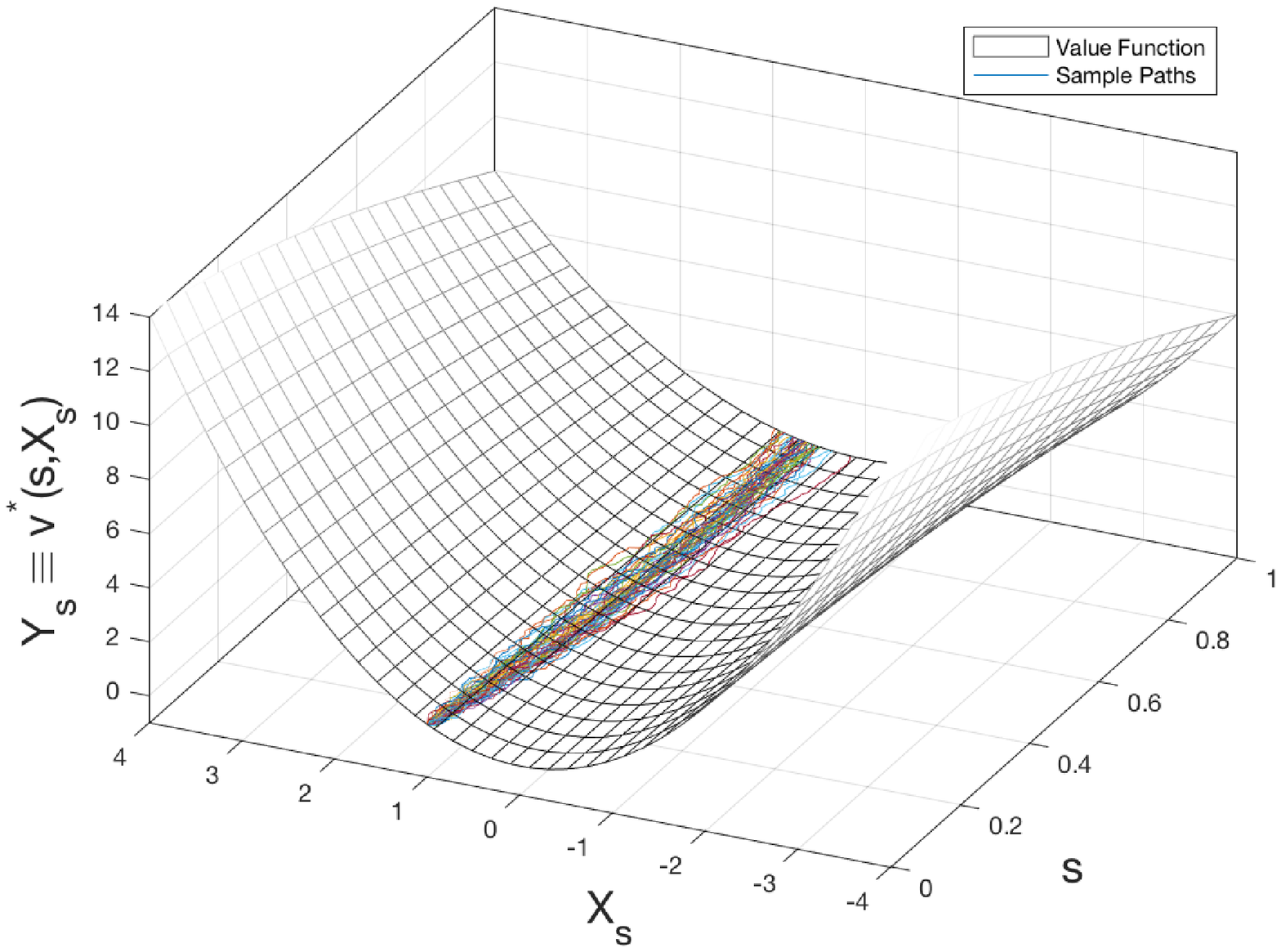}}
\caption{Surface of the optimal value function 
for Example~\ref{eg:toyexample}.
50~sample paths of drifted forward-backward SDEs (\ref{eq:driftfsde})-(\ref{eq:driftbsde}) are compared with different drifts. 
The initial condition is chosen as $x_0 = 1$.}
\label{fig:50SamplePaths}
\end{figure}

\begin{figure}
\centering
\subfloat[$K_s = f^{\pi^*}_s$, the same case as in Fig.~\ref{fig:50SamplePaths}~(c).]{\includegraphics[width=0.95\linewidth]{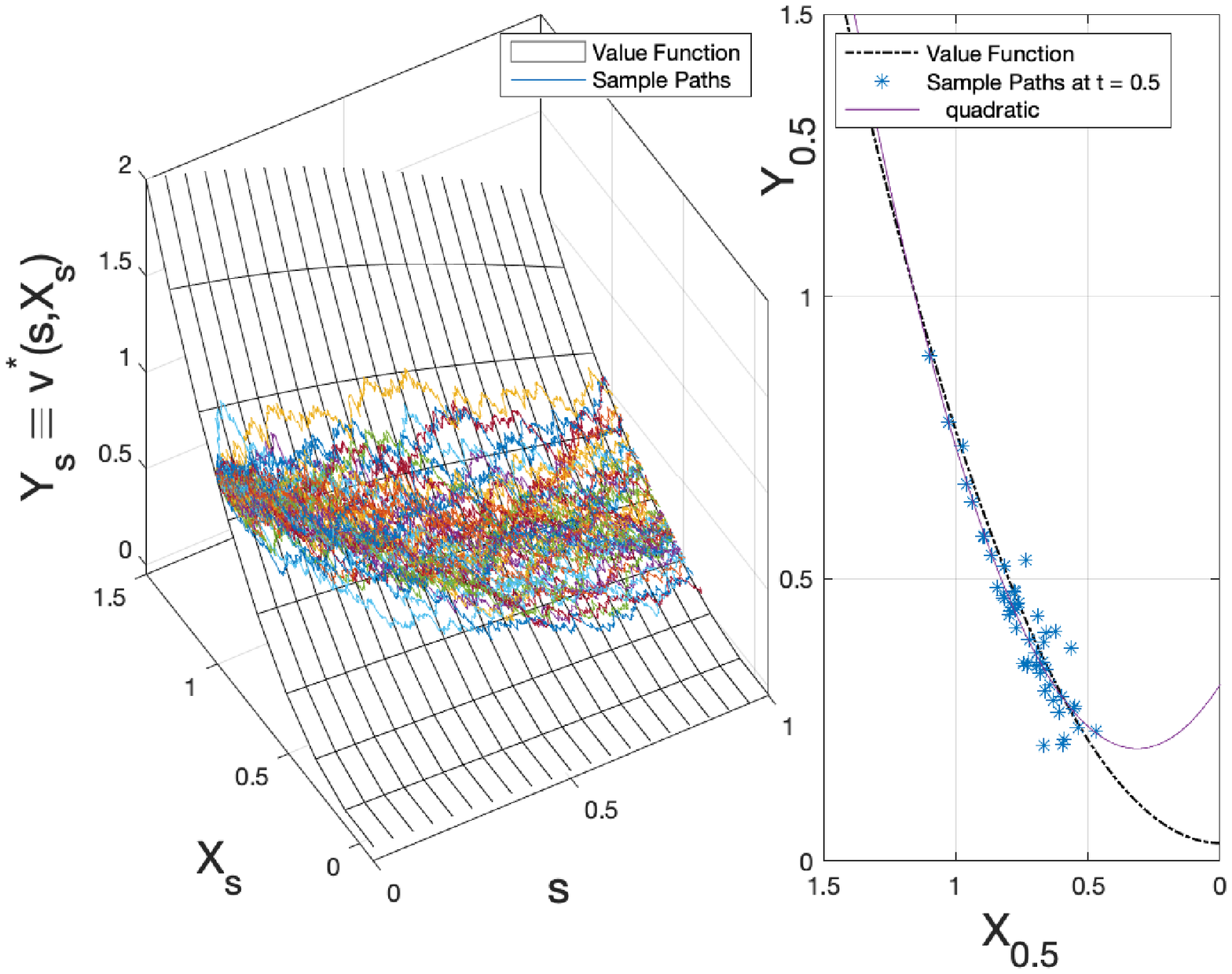}}
\\
\subfloat[][$K_{s}=\begin{cases}
f(X_{s},{\pi^{*}}(X_{s})), & \text{w.p. }0.8\text{,}\\
f(X_{s},-30), & \text{w.p. }0.1\text{,}\\
f(X_{s},30), & \text{w.p. }0.1\text{,}
\end{cases}$]{\includegraphics[width=0.95\linewidth]{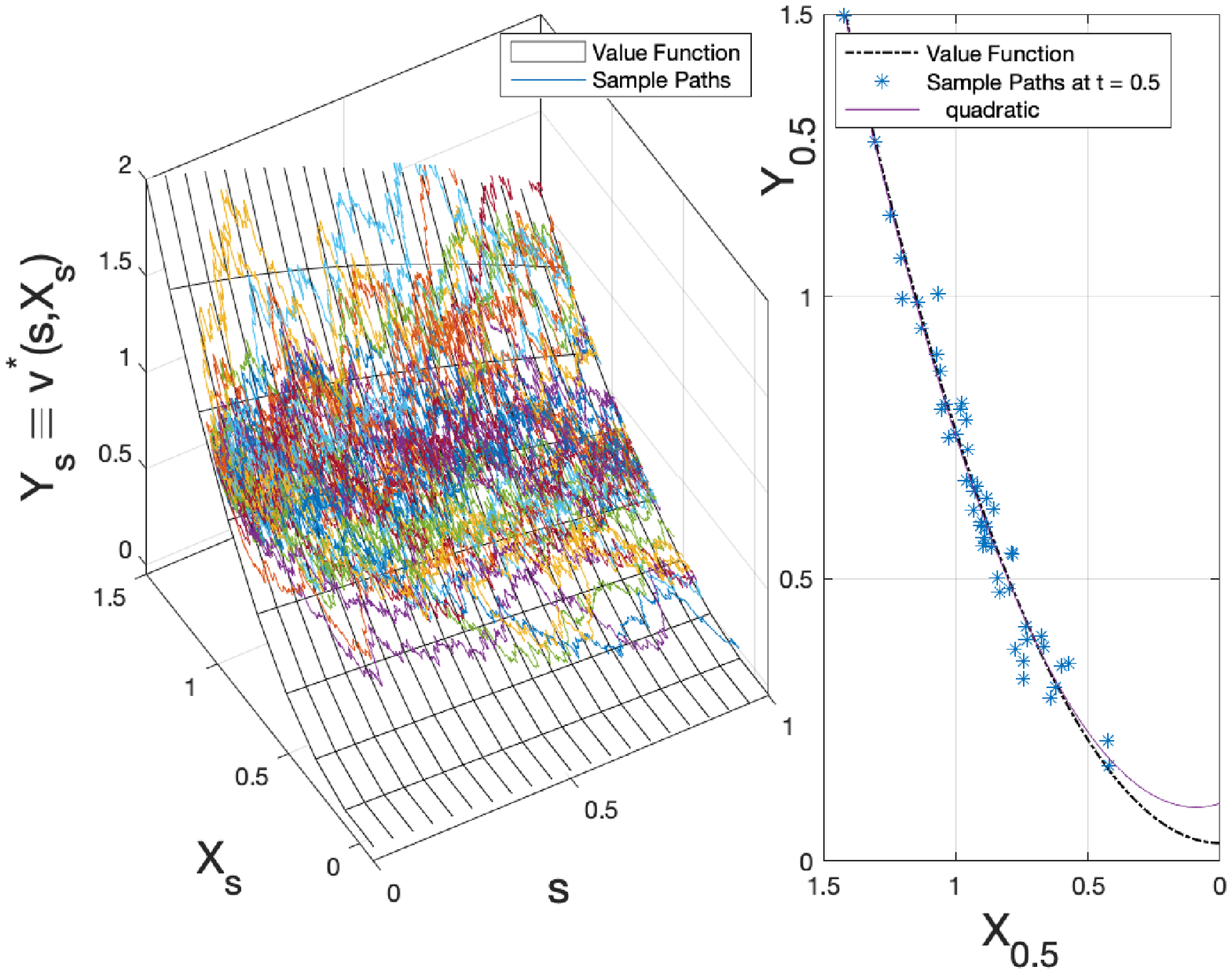}}
\caption{The implementation of the drift corresponding to (a) optimal input, and (b) optimal input with $10\%$ randomization, for the same system as 
in Example \ref{eg:toyexample}.}
\label{fig:Randomization}
\end{figure}

\begin{example}\label{eg:toyexample} 
Consider the optimal control of the scalar linear system 
\begin{align*}
\dX_s = (X_s + u_s) \, \ds + \sfrac{1}{5} \, \dW_s \text{,}
\end{align*}
    whose value function is defined as the solution of
\begin{align*}
    V^* (t,x) = \inf_{u[t,T]}  \E \left[ \mbox{$\frac{1}{2}$} (X_T)^2 
    + \int_t^T \mbox{$\frac{1}{2}$} u_s^2 \, \ds \right] \text{.}
\end{align*}
    It can be verified that the value function is analytically expressed as
\begin{align*}
    V^*(t,x) &= \alpha_t x^2 + \beta_t \text{,} 
    \end{align*}
 where
 \begin{align*}   
    \alpha(t) &:= (e^{-2(1-t)} + 1)^{-1} \text{,} \\
    \beta(t) &:= \mbox{$\frac{1}{50}$} \Big(\log\Big(\frac{1}{2}+ \mbox{$\frac{1}{2}$}   e^{2(1-t)}\Big)+2(1-t)\Big) ,
\end{align*}
with corresponding optimal policy \mbox{$\pi^*(t,x) = -2 \alpha_t x$}.
For this example, let the target policy be the optimal policy $\mu = \pi^*$, which can be computed analytically.
For any adapted process $K_s$, the associated drifted FBSDE is 
\begin{align*}
    \dX_s &= K_s \, \ds + \sfrac{1}{5} \, \dW_s^{\P},\\
    \dY_s &= -\left( \sfrac{1}{2} \mu^2_s + 5 (X_s + \mu_s) - 5 Z_s K_s \right) \, \ds +
    Z_s \, \dW^{\P}_s \text{,}
\end{align*}
    with boundary conditions
\begin{align*}
    X_0 &= x_0, \hspace{10pt} Y_T = \mbox{$\frac{1}{2}$} (X_T)^2  \nonumber \text{,}
\end{align*}
    where we choose $x_0 = 1$.
We also have $Z_t = ({2}/{5}) \alpha_t X_t$.
The correspondence between the value function and the solutions to the above FBSDE for
    three different choices of $K_s$, are illustrated in Fig. \ref{fig:50SamplePaths}.
The drift term in (c) generates a distribution for $X_s$ which matches the system
    guided by the optimal control policy.

Regardless of the drift value chosen, the process 
    $\begin{bmatrix} X_s & Y_s \end{bmatrix}^\top$ lies on the surface characterized
    by the function $V^*$.
Since, in general, function approximation has higher accuracy when interpolating in a region of dense samples compared to extrapolating in a region with no samples, the case in Fig.~\ref{fig:50SamplePaths}(c) is more desirable from an optimal control perspective, compared to Figs.~\ref{fig:50SamplePaths}(a) and \ref{fig:50SamplePaths}(b), since the samples have a correspondence with the optimal trajectories.

However, as illustrated in Fig.~\ref{fig:Randomization}, there are other selections for the sampling policy outperforming the optimal control policy, yielding better function approximation.
In this example, the function approximation at $t=0.5$, is illustrated on the right hand side of Fig.~\ref{fig:Randomization}. 
In particular, Fig.~\ref{fig:Randomization}(b) illustrates that such a randomized optimal policy explores a larger region compared to a pure implementation of the optimal drift in Fig.~\ref{fig:Randomization}(a), thus resulting in a more accurate approximation of the value function. 
In other words, a broader exploration of the state space contributes to better function approximations in the presence of numerical error that builds from recursive function approximation during the backward pass. 
\hfill $\square$
\end{example}

Example~\ref{eg:toyexample} and Fig.~\ref{fig:50SamplePaths} illustrate the link between
    the drifted FBSDE and the value function.
We can interpret this result in the following sense.
We can pick an arbitrary process $K_s$ to be the drift term, which generates a distribution for the forward process $X_s$ in the corresponding measure $\P$. 
The BSDE yields an expression for $Y_t$ using the same process $W^{\P}_s$ used in the FSDE. 
The term $Z^{\top}_s D_s$ acts as a correction in the BSDE to compensate for changing the drift of the FSDE. 
We can then use the relationship \eqref{eq:thmcontinfk2} to solve for the value function $V^\mu$, whose conditional expectation can be evaluated in $\P$. 

It should be noted that $K_s$ need not be a deterministic function of the random variable $X_s$, as is the case with $f^\mu_s$. 
For instance, it can be selected as  \mbox{$K_s(\omega) = \kappa(s,X_s(\omega),\omega)$} for some appropriate function $\kappa$, producing a non-trivial joint distribution for the random variables $(X_t,K_t)$.

A remarkable feature of the off-policy FBSDE formulation is that the forward 
    pass is decoupled from the backward pass, that is, the evolution of the forward SDE 
    does not explicitly depend on $Y_s$ or $Z_s$
    (whereas in the Stochastic Maximum Principle formulations 
    (see, e.g., \cite[Chapter 3]{yong1999stochastic}) the decoupling is irremovable). 
This feature forms the basis of FBSDE numerical investigations of stochastic optimal control  
    \cite{Bender2010, exarchos2018stochastic}.
    The significant difference of 
    Theorem~\ref{thm:driftfbsde} in comparison to those results 
    is that the focus is shifted here 
    from the solution of \eqref{eq:hjbpde} towards the broader class of functions satisfying 
    the \eqref{eq:hjpde}.
This provides a stronger case for policy iteration methodologies, because the theory does not
    require, or expect, $\mu$ to be an optimal policy, as is in 
    \cite{Bender2010, exarchos2018stochastic}.
Although not evaluated in this work, $\mu$ can be chosen according to design specifications
    other than estimating the optimal policy, such as to ensure the current policy
    is nearby the previously estimated policy.

\subsection{Interpretation of On-Policy Value Functions} \label{sec:interpretation}

When $V^\mu$ is a non-classical viscosity solution, $\partial_x V^\mu$ may not exist everywhere
    and the expression \eqref{eq:zv} might not hold everywhere
    (see, e.g., the discussion in \cite[p.46-47]{el1997backward}).
Nonetheless, we can modify Theorem~\ref{thm:driftfbsde} to yield a similar result.

\begin{proposition} \label{cor:onpolicy_exists}
Let Assumptions~\ref{ass:optviscosity}-\ref{ass:onpolicy_easy} hold, and let $V^\mu$ be the viscosity
    solution of \eqref{eq:hjpde}.
Then, there exists a function \mbox{$G : [0,T] \times \R^n \rightarrow \R^n$} such that
    $\wt{Z}_s = \sigma^\top_s G(s,\wt{X}_s)$ $\Q$-a.s., where
    $(\wt{X}_s, \wt{Y}_s, \wt{Z}_s)$ is the solution to the FBSDE (\ref{eq:fsde_alternate})-(\ref{eq:bsde_alternate}).
Let $\mu(t,x) := \mu(t,x,G(t,x))$ and define
    $f^\mu$, $\ell^\mu$ accordingly, and  replace \eqref{eq:zv} with
\begin{align}
    Z_s = \sigma_s^\top G(s,X_s) \text{,} \quad \text{a.e.} \;  s \in [0,T] \text{.} \label{eq:zv_new}
\end{align}
    Then, the conclusions of Theorem~\ref{thm:driftfbsde} hold subject to these modifications.
\end{proposition}

\begin{proof}
    The existence of $G$ with the relevant property follows from \cite[Theorem~4.1]{el1997backward}.
    It is easy to verify that \eqref{eq:hmuequals} is satisfied under the proposed modifications.
    The rest of the proof follows similarly to the proof of Theorem~\ref{thm:driftfbsde}.
\end{proof}

We now use this proposition to show the following result. 

\begin{corollary} \label{eq:onpolicy_interpretation}
Let Assumptions~\ref{ass:optviscosity}-\ref{ass:onpolicy_easy} hold, let $V^\mu$ be the viscosity
    solution of \eqref{eq:hjpde}, and let $G, \mu, f^\mu, \ell^\mu$ be as in Proposition~\ref{cor:onpolicy_exists}, where
    $G = \partial_x V^\mu$. 
    If also Assumption~\ref{ass:onpolicy_strong} holds, 
then the on-policy value function has the interpretation
\begin{align}
    V^\mu(t,x) &= \E \bigg [ \int^T_t \ell^\mu(s,X_s) \, \ds + g(X_T) \bigg ] \text{,}
\end{align}
    where,
\begin{align}
    \dX_s &= f^\mu(s,X_s) \, \ds + \sigma(s,X_s) \, \dW_s, \quad X_t = x \text{.}
\end{align}

\end{corollary}

\begin{proof}
    Letting $K_s = f^\mu_s$, Proposition~\ref{cor:onpolicy_exists} yields the result in
    equations (\ref{eq:thmcontinfk2})-(\ref{eq:bsdediff2}).
\end{proof}

Note that while $G = \partial_x V^\mu$ cannot be guaranteed for all viscosity solutions of \eqref{eq:hjpde},
    it does hold for smooth convergent approximations of $V^\mu$
    (see, for example, proof of Theorem~4.5 in~\cite[Chapter~7]{yong1999stochastic}).

\subsection{Local Entropy Weighing}

As discussed in Section~\ref{sec:OffPolicy}, the disentanglement of the forward sampling from the backward function approximation provides the opportunity to employ broad sampling schemes to cover the state space with potential paths. 
However, fitting a value function broadly to a wide support distribution might degrade the quality of the function approximation since high accuracy of function approximation is more crucial in those parts of the state space that are in proximity to the optimal trajectories. 
Once forward sampling has been performed and some parts of the value function have been approximated, we can apply a heuristic in which sample paths closer to optimal trajectories are weighted more so as to concentrate value function approximation accuracy in those regions.

To this end, we propose using a bounded heuristic random variable $\rho_t$ 
    to produce a new measure $\pR_t$, the weighted counterpart to $\P_t$,
    defined as 
\begin{align}
    \mathrm{d}\Q_t = \Theta_{0,t}^{\Q|\P} \mathrm{d}\P_t \text{.}
\end{align}
In order to avoid underdetermination of the regression by concentrating a single or few samples, 
    we select $\pR_t$ as 
\begin{align}
    \pR_t \in \argmin_{\pR_t} \big \{ \E_{\pR_t}[\rho_t] 
    + \lambda \mathcal{H}(\pR_t \| \P_t) \big \} \text{,} \label{eq:MinRelativeEntropyMeasure}
\end{align}
with $\lambda > 0$, a tuning variable, 
and 
\begin{align}
    \mathcal{H}(\pR_t \| \P_t) = \E_{\pR_t}  \bigg[ \log \bigg ( 
    \frac{\mathrm{d} \pR_t}{\mathrm{d} \P_t} \bigg )  \bigg] \text{,}
\end{align}
 is   the relative entropy of $\pR_t$ which takes its minimum value when $\pR_t = \P_t$, 
    the distribution in which all sampled paths have equal weight. 


The minimizer of \eqref{eq:MinRelativeEntropyMeasure}, which balances between the value of $\rho_t$ 
    and the relative entropy of its induced measure, 
    has a solution $\pR^*_t$, given by \cite[p.~2]{theodorou2012relative}
\begin{gather}
    \mathrm{d} \pR^*_t 
    = \Theta^{\pR|\P}_t 
    \mathrm{d} \P_t  \text{,} \qquad
    \Theta^{\pR|\P}_t := 
    \frac{ \exp(-\sfrac{1}{\lambda} \rho_t) }{ \E_{\P_t}[\exp(-\sfrac{1}{\lambda} \rho_t)] } 
    \text{.} \label{eq:changemeas} 
\end{gather}

Henceforth, for simplicity, we let $\pR_t$ refer to this minimizer $\pR^*_t$. 
During numerical approximation we can interpret the weights as a \textit{softmin} operation over paths according to this heuristic, a method often used in the deep learning literature~\cite{Goodfellow-et-al-2016}.

\begin{theorem}\label{thm:weight}
Assume $\rho_\tau$ is selected such that $W^\P_s$ is Brownian on the interval $[t,\tau]$
    with respect to the induced measure $\pR_\tau$.
It then holds that
\begin{align}
    Y_t &= \E_{\pR_\tau}[\wh{Y}_{t,\tau}| X_t] = V^\mu(t,X_t), 
    & \pR_\tau\text{-a.s.} \text{,}
    \label{eq:thmcontinfk3}
\end{align}
    where $\wh{Y}_{t,\tau}$ is defined in \eqref{eq:bsdediff2}.
Furthermore, the minimizer $\phi^*$ of the optimization problem
\begin{align}
    &\inf_{\phi \in L^2} 
    \E_{\pR_\tau}[(\wh{Y}_{t,\tau} - \phi(X_t))^2] \nonumber \\
    &\quad = \inf_{\phi \in L^2} 
    \E_{\P_\tau}[\Theta^{\pR|\P}_\tau (\wh{Y}_{t,\tau} - \phi(X_t))^2] \text{,}
    \label{eq:thmcontinfk4}
\end{align}
    over $X_t$-measurable square integrable variables $\phi(X_t)$ coincides with
    the value function $\phi^*(X_t) = V^\mu(t,X_t)$.
\hfill $\square$
\end{theorem}

\begin{proof}
First, note that $X_s$, $Y_s$, $Z_s$, $W_s^{\P}$, and $\wh{Y}_{t,\tau}$ are 
    $\F_\tau$-measurable for $s, t \in [0,\tau]$.
Thus, \eqref{eq:yv} and 
\begin{align}
    \wh{Y}_{t,\tau} &= Y_t 
        - \int_t^\tau Z^{\top}_s \dW_s^{\P}  , \label{eq:yhtzdw}
\end{align}
    hold $\P_\tau$-a.s..
We now show that they hold $\pR_\tau$-a.s. as well.
Since $\E_{\P_\tau}[\Theta^{\pR|\P}_\tau] = 1$ and $\P_\tau$ is a probability measure,
    then $\pR_\tau$ is also a probability measure.
Furthermore, since $\rho_\tau$ is bounded, $\Theta_\tau > 0$ $\P_\tau$-a.s..
It follows that $\pR_\tau$ and $\P_\tau$ are equivalent measures because they reciprocally
    have strictly positive densities \cite[Chapter~10, Remark~10.4]{pascucci2011pde}. 
The proof of Theorem~\ref{thm:driftfbsde} shows that since these measures are equivalent,
    \eqref{eq:yv} and \eqref{eq:yhtzdw} hold $\P_\tau$-a.s. if and only if they also hold
    $\pR_\tau$-a.s..
Since $W^\P_s$ is Brownian in $\pR_\tau$ over the integral, the second term in the
    right hand side of \eqref{eq:yhtzdw} will drop out when taking the conditional
    expectation $\E_{\pR_\tau}[\wh{Y}_{t,\tau}|X_t]$, yielding \eqref{eq:thmcontinfk3}.

Equation \eqref{eq:thmcontinfk4} is a result of the $L^2$-projective properties
    of conditional expectation \cite[Chapter~10.3, Property~11]{resnick2003probability}
    applied to \eqref{eq:thmcontinfk3}, followed by
    a change of measure applied to the expected value integration with \eqref{eq:changemeas}.
\end{proof}

In the following section, we approximate the minimization of the right hand side of 
    \eqref{eq:thmcontinfk4} over parameterized value function models 
    to obtain an estimate of the value function. 


To summarize, in this section we introduced three measures: (a) $\Q$, the measure associated with the target policy $\mu$ for the value function $V^\mu$, (b) $\P$, the sampling measure used in the forward pass to explore the state space, and (c) $\pR_\tau$, the local-entropy weighted measure used in the backward pass to control function approximation accuracy.

In the this section we have provided  the theoretical results justifying our approach 
    for the case of continuous-time stochastic processes.
In the next section we discuss how these each of these measures are represented numerically.

\section{Branching Path LSMC} \label{sec:branchpathlsmc}



In this section we propose
    a novel discrete-time, finite-dimensional numerical scheme
    to produce the FSDE distribution, along with a procedure
    to solve for the value function in a backward pass using the BSDE.
The FSDE distribution is represented as a branching-path tree and the BSDE is used to produce estimators,
    stepping backwards along each of the branching paths, to estimate the value function parameters
    using LSMC regression.
Later on, in Section~\ref{sec:fbrrt}, we propose particular choices for the drift process and the
    heuristic weight function used in the proposed
    FBRRT-SDE numerical method.

Henceforth, we assume a discrete-time partition of the interval $[0,T]$,
    $(t_0 = 0, \ldots, t_i = (\Delta t) i, \ldots, t_N = T)$,
    for some partition length $\Delta t$.
For brevity, we abbreviate $X_{t_i}$ as $X_i$ and similarly for most  other variables.

\begin{figure}
    \centering
    \subfloat[Optimal Distribution\label{fig:distopt}]{
        \includegraphics[width=0.47\linewidth]{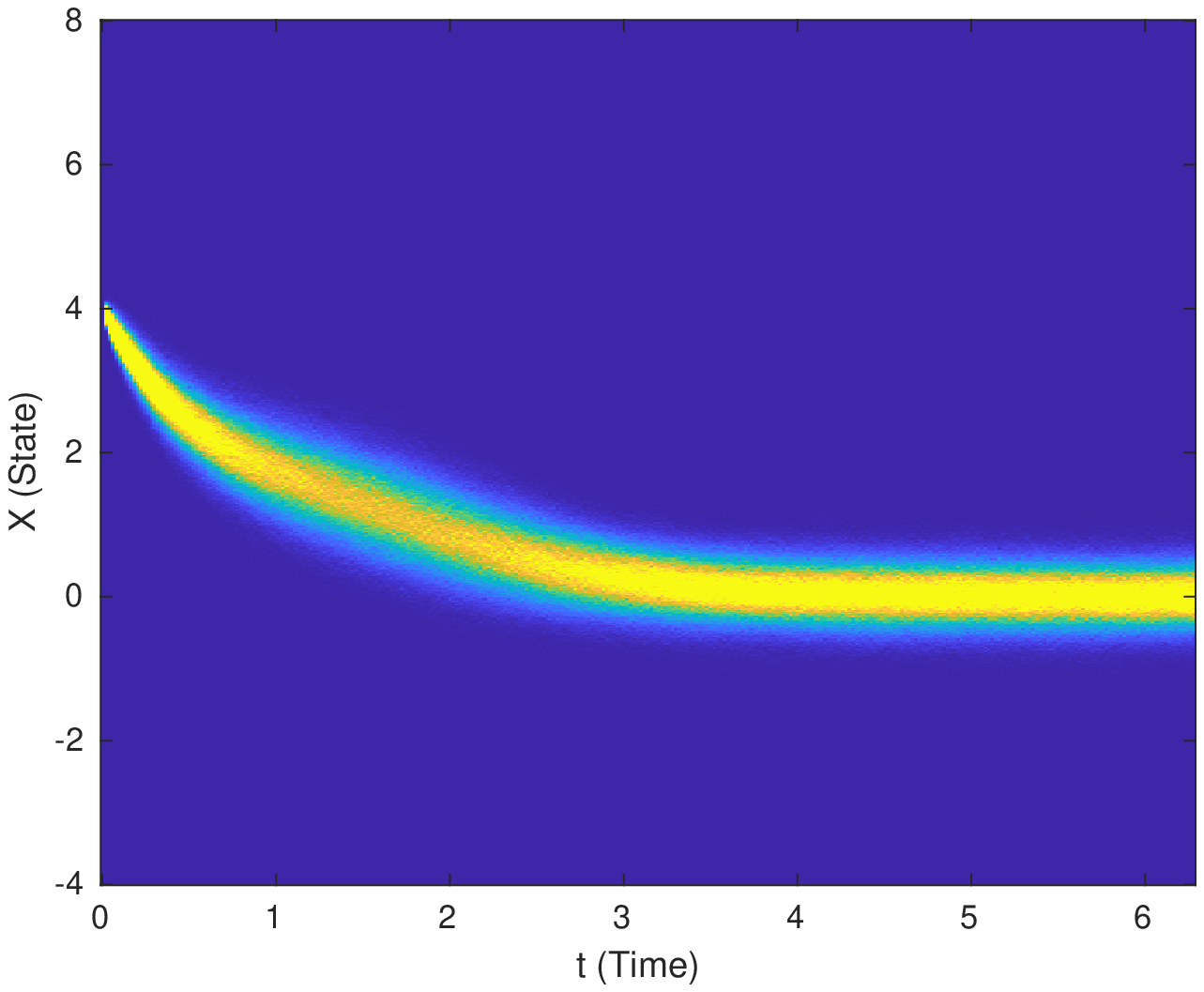}} 
    \subfloat[Parallel-Sampled Suboptimal\label{fig:distsubopt}]{
        \includegraphics[width=0.47\linewidth]{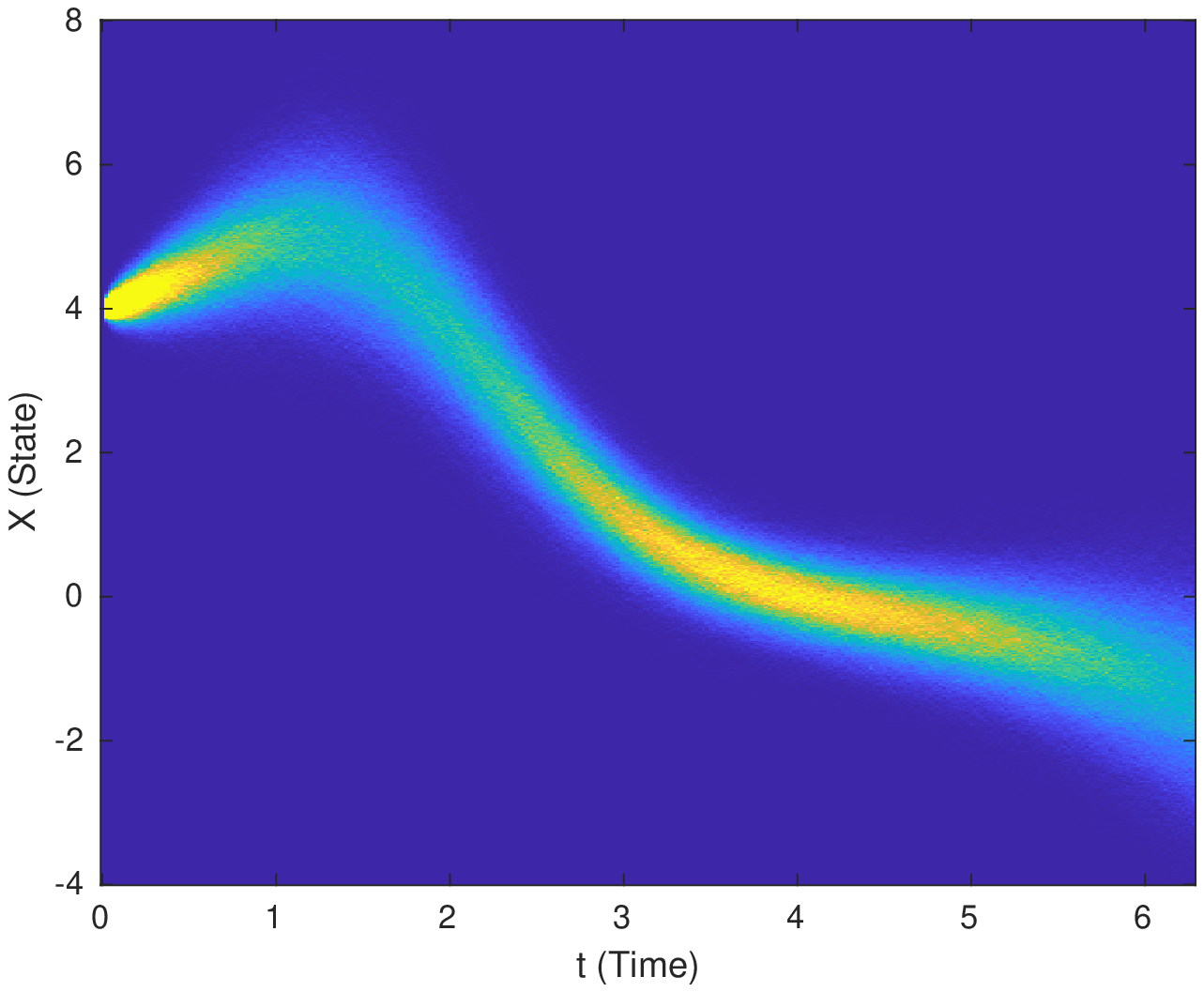}} \\
    \subfloat[RRT-Sampled ($\P$)\label{fig:distrrt}]{
        \includegraphics[width=0.47\linewidth]{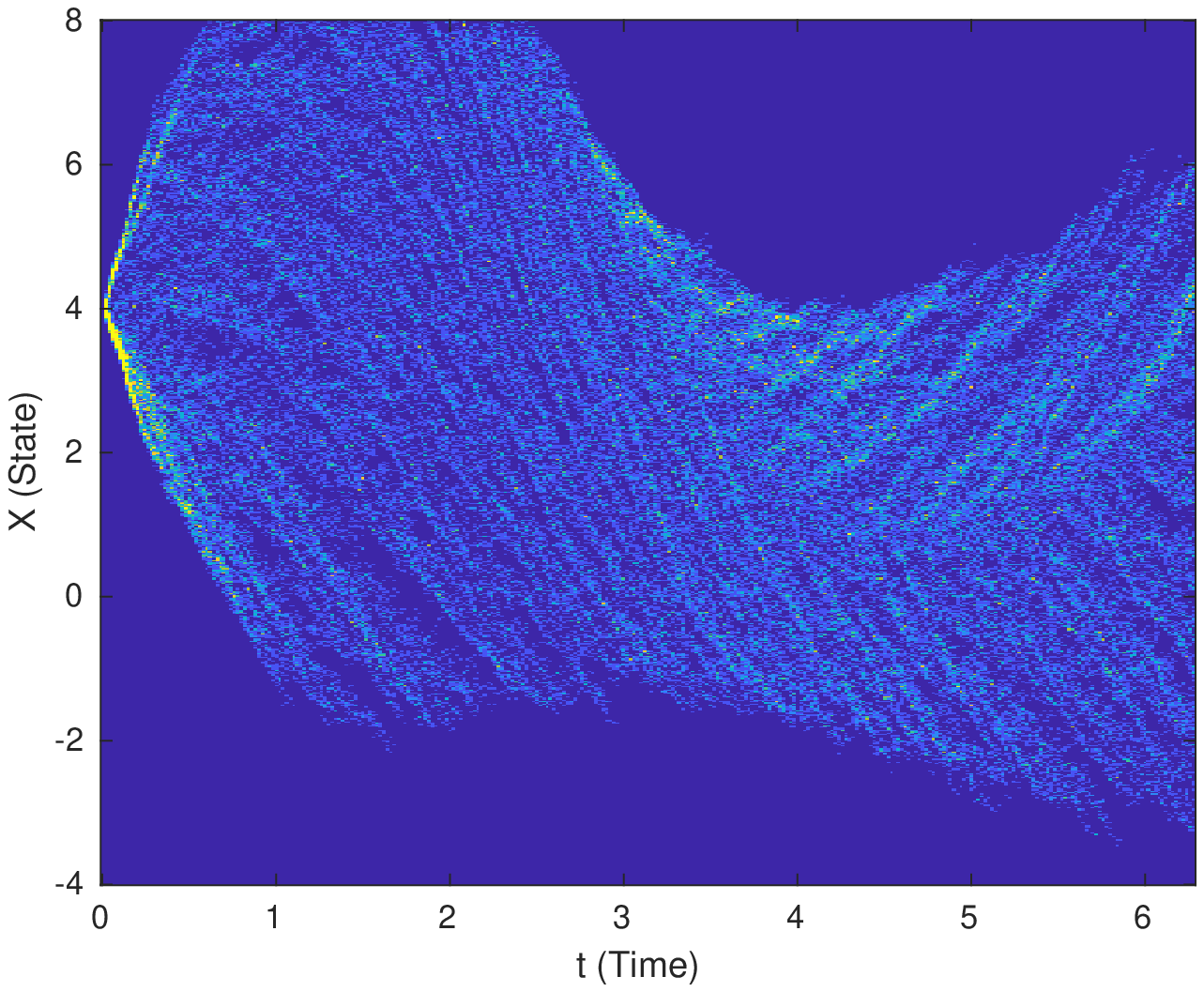}} 
    \subfloat[RRT-Sampled, Weighted ($\pR$)\label{fig:distweight}]{
        \includegraphics[width=0.47\linewidth]{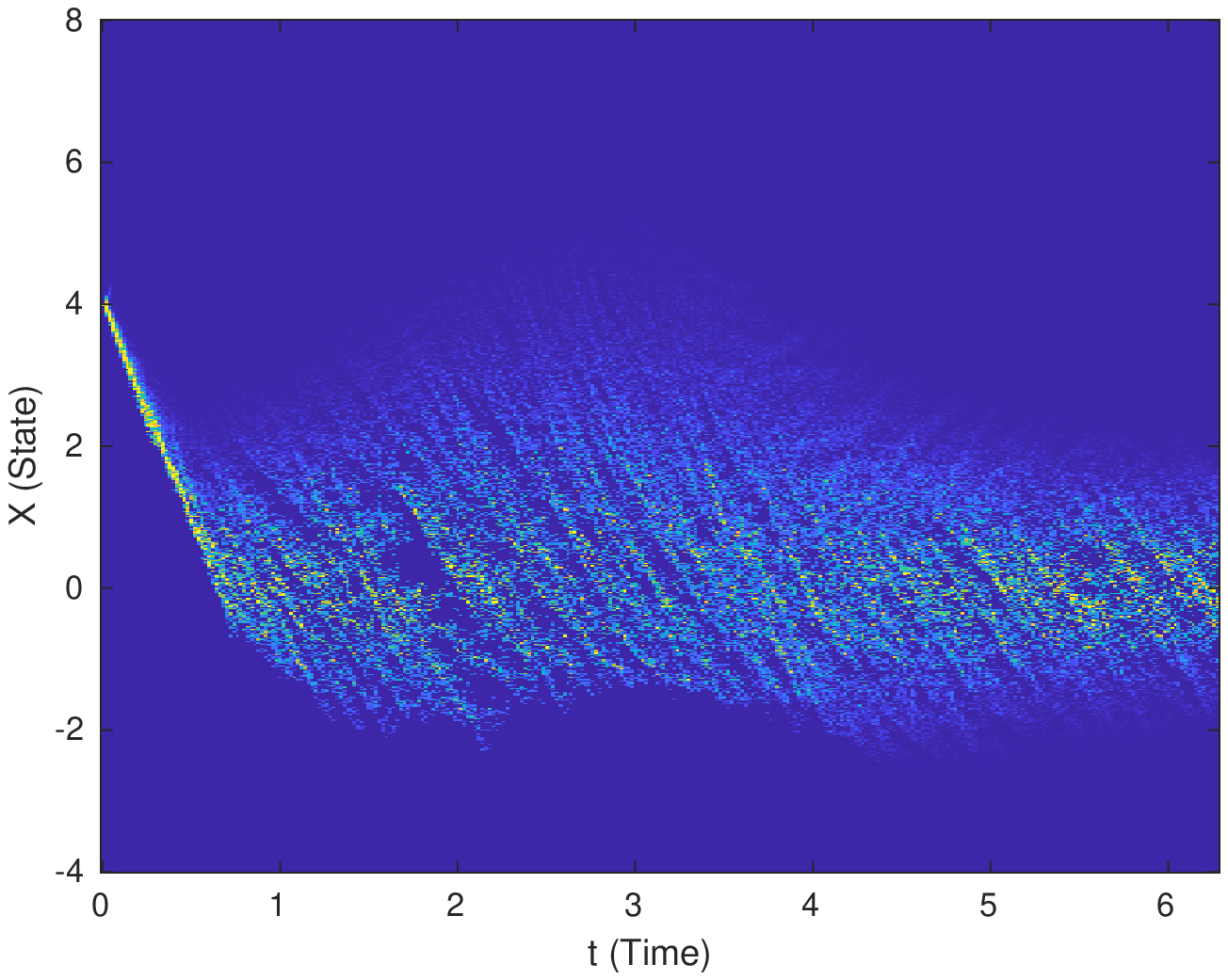}}
    \caption{Heatmap of different measure distributions for a $1$-dimensional SOC problem, 
        illustrating how RRT-sampling and local-entropy weighing can accelerate
        discovery of the optimal distribution.}
    \label{fig:distributions}
\end{figure}

Fig.~\ref{fig:distributions} motivates our approach, illustrating how the method is able to 
    rapidly find the optimal distribution. 
An on-policy method assumes knowledge of an initial suboptimal control policy, sampled as  in Fig.~\ref{fig:distributions}\subref{fig:distsubopt} using the approach in \cite{Exarchos2018} and the suboptimal value function is solved in that distribution. 
The on-policy method
requires iterative improvement of the policy to produce a distribution which overlaps with the optimal distribution.
However, if we begin with a sampling measure which broadly explores the state space as in Fig.~\ref{fig:distributions}\subref{fig:distrrt}, we can produce an informed heuristic which weighs this distribution as in Fig.~\ref{fig:distributions}\subref{fig:distweight}, so that the function approximation is concentrated in a near-optimal distribution.
Thus, often, we need only one iteration to get a good approximation of the optimal value function and policy.

In Section~\ref{sec:SDEBranched} we summarize the construction of a stochastically sampled tree, as
    a generalized data structure
    to approximate the FSDE distribution over the partition $t_0, t_1,\ldots,t_N$, and 
then, in Section~\ref{sec:McKeanMarkov}, we
demonstrate  how this data structure can be interpreted as a series of McKean-Markov path measures
    $\{\fP_i\}_{i=0}^N$ to approximate the forward sampling distributions. 
    Finally, in Section~\ref{sec:pathlsmc}, we discuss how these measures can be used in the backward pass to approximate
    the BSDE solution by estimating the optimal value function.

\subsection{Forward SDE Branched Sampling}
\label{sec:SDEBranched}

We begin by discussing the construction of a tree data structure $\mathcal{G}$ 
    representing the FSDE \eqref{eq:driftfsde}.
In this section we only describe how edges of $\mathcal{G}$ are added and what data is stored.
    Later, in Section~\ref{sec:kinorrt}, we propose a specific methodology for selecting nodes
    for expansion and choosing the drift value.
The tree is initialized with a root node at the initial starte $x_0$ and is
    constructed asynchronously as long as new nodes and directed edges are added using the
    following procedure.
    
Let \mbox{$x^\text{parent}_i \in \R^n$} be a state node in the tree at time $i$ selected for expansion, as
    the parent of a new edge.
The drift $k_i \sim \kappa(x^\text{parent}_i,\{x^k_i\}_k)$ 
(representing the random variable $K_i$)
is sampled from some random function $\kappa$
    which can depend on both the state and the distribution of the nodes at that time.
Independently, the noise is sampled $w_i \sim \mathcal{N}(0,\Delta t I_n)$.
The child state node is computed using an Euler-Maruyama SDE step approximation of the
    FSDE \eqref{eq:driftfsde},
\begin{align}
    x^\text{child}_{i+1} = x^\text{parent}_i + k_i \Delta t + \sigma(t_i, x^\text{parent}_{i}) w_i
    \text{.}
    \label{eq:eulermaruyama}
\end{align}
The edge $(x^\text{parent}_i,d^\text{data}_i,x^\text{child}_{i+1})$ is added to the tree, where
    \mbox{$d^\text{data}_i = (k_i, w_i, \ldots)$} is the data attached to the edge. 
    A new parent 
    can then be selected for expansion, including selecting the same parent again.
\mbox{Fig.~\ref{fig:parallelvsweighted}~(a-b)} illustrates the branching tree data structure.

\subsection{McKean-Markov Measure Representation}
\label{sec:McKeanMarkov}

\begin{figure}
    \centering
    \subfloat[][Branch-Sampled ($\fP_{i}$)]{%
        \includegraphics[width=0.43\linewidth]{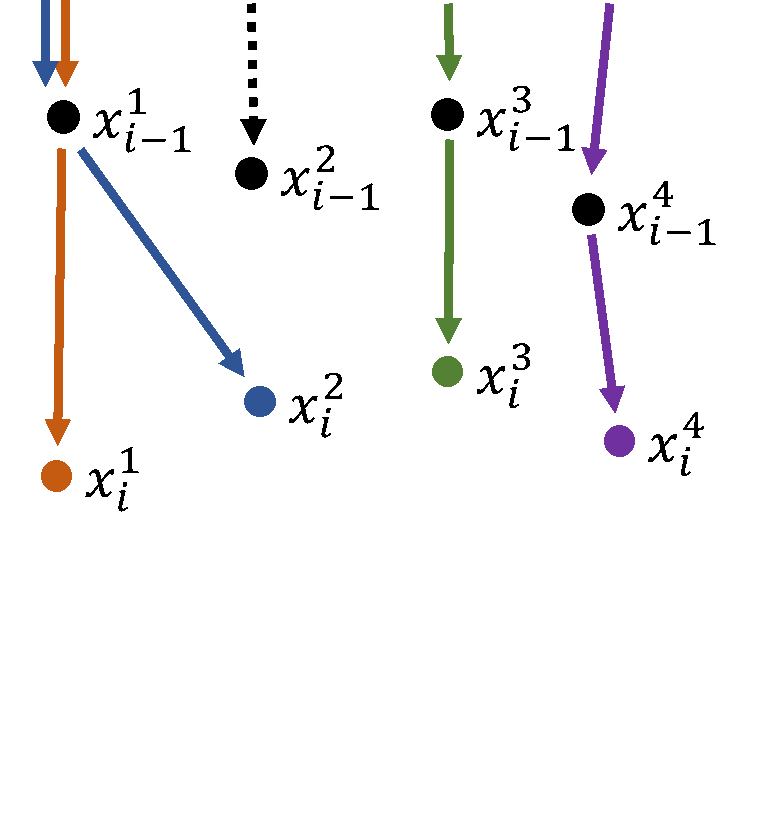} \label{fig:branchsampi}} 
    \subfloat[][Branch-Sampled ($\fP_{i+1}$)] {%
        \includegraphics[width=0.43\linewidth]{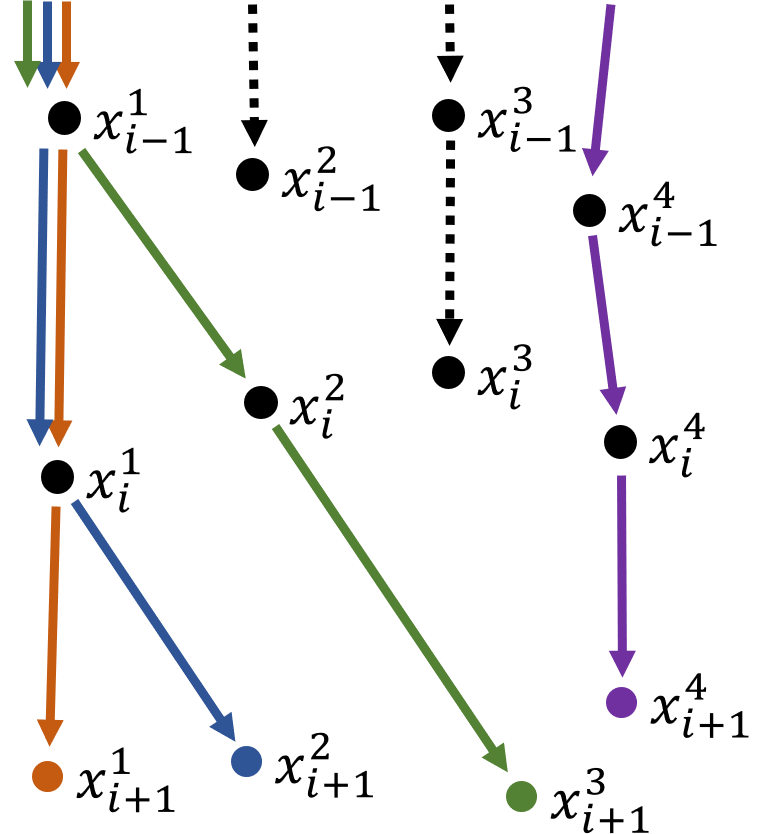} \label{fig:branchsampip1}} \\
    \subfloat[][Parallel-Sampled ]{%
        \includegraphics[width=0.43\linewidth]{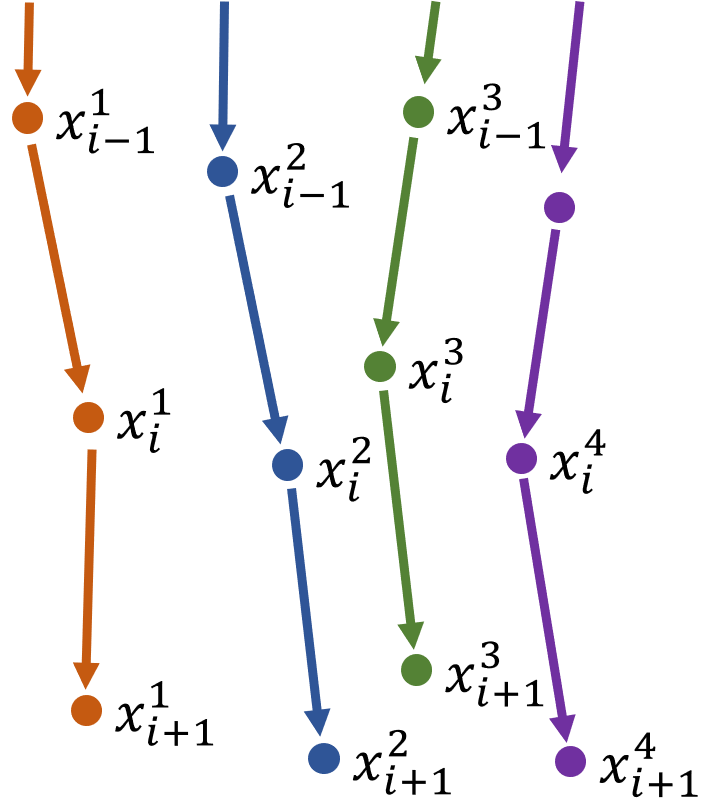} \label{fig:parallel}} 
    \subfloat[][Weighted-Branched ($\bR_{i+1}$)]{%
        \includegraphics[width=0.43\linewidth]{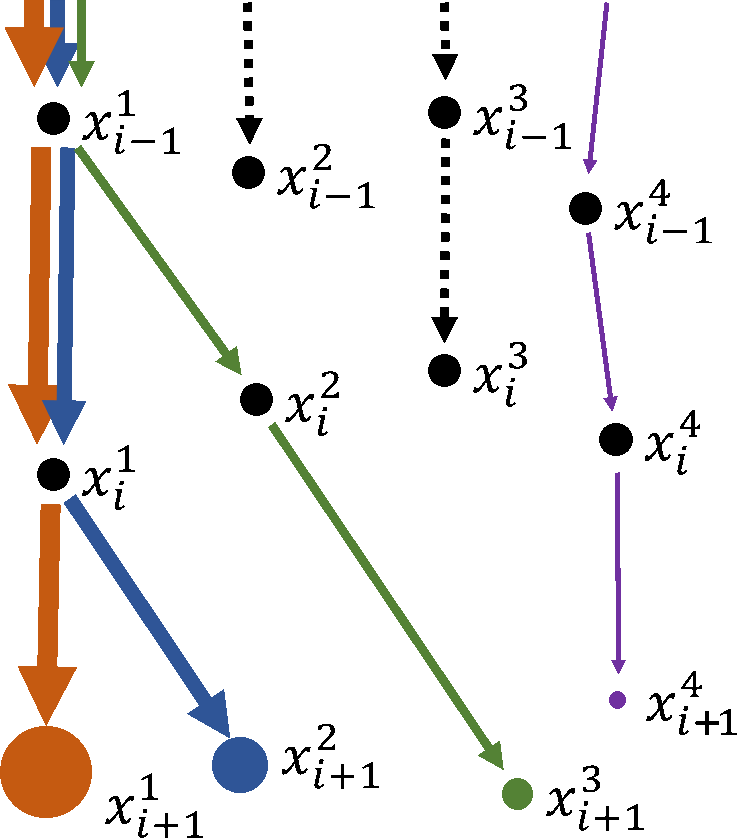}} 
    \caption{ \textbf{(a-b)} Illustration of how the branch sampled measures are represented
        based on the underlying data structure.
        The colored paths show the collection of paths representing the respective measure.
        Dotted lines represent edges in the data structure which are not included in the path measure
        for that time step.
        \mbox{\textbf{(c-d)}} Comparison of the unweighted parallel-sampling method from previous approaches
    to the proposed weighted and branch-sampled method.}
    \label{fig:parallelvsweighted}
\end{figure}

We approximate the continuous-time sampling distributions with discrete-time 
    McKean-Markov branch sampled paths, as presented in \cite{del2013mean}.
The tree data structure $\mathcal{G}$
    represents a series of path measures $\{\fP_i\}_{i=0}^N$,
    each approximating the distribution
$
    \fP_i \approx \P_{t_i} \circ \xi_i^{-1} ,
$
    where $\xi_i$ is the discrete-time random path defined as
$
    \xi_i := (X_0, \mathcal{D}_0, X_1, \cdots, \mathcal{D}_{i-1}, X_i),
$
    and $\xi_i^{-1}$ is the inverse map from events on the path space to events on the sample
    space $\Omega$ \cite[Chapter~3]{resnick2003probability}.
Here, we use $\mathcal{D}_i$ to refer to the set of 
    random variables associated 
    with the edges of the tree, 
    including $K_i$ and $W^\P_i$.
The empirical measure approximations are defined as
\begin{align}
    \fPi &:=
    \frac{1}{M} \sum_{j=1}^M \delta_{\xi_i^j} \text{,}
\end{align}
    where $\delta$ is the Dirac-delta measure acting on sample paths
\begin{align}
    \xi_i^j &:= (x_{0,i}^j, d_{0,i}^j, 
    x_{1,i}^j, d_{1,i}^j, 
    \ldots, d_{i-1,i}^j, x_{i,i}^j).
\end{align}
The notation $x_{m,i}^j$ indicates that this element is the sample of a random variable
    $X_m$ that is the ancestor of sample $x_{i,i}^j$ in the path $\xi_i^j$, and similarly for
    the edge variables $\mathcal{D}_m$.
Each node in the tree $x^j_i$ (alternatively called a particle)
    is associated with a unique path $\xi_i^j$ whose final term
    is $x_{i,i}^j = x^j_i$.
Figures~\ref{fig:parallelvsweighted}(a)-(b)
    illustrate how each colored node at a particular time step is associated with its 
    matching colored path,
    and that all of these paths collectively constitute the path measure.

It is worth noting that in this construction there is no requirement for
    $\fP_i$ and $\fP_{i+1}$ to agree over the interval $[0,t_i]$.
This property is illustrated by the fact that, for example, the path ending at $x^3_i$ in
    \mbox{Fig.~\ref{fig:parallelvsweighted}(a)} 
    is represented in $\fP_{i}$ but not represented in $\fP_{i+1}$ in 
    \mbox{Fig.~\ref{fig:parallelvsweighted}(b)}.
To see why such constructions are permissible in the proposed numerical scheme, 
    notice that in the backward step 
    (e.g., in Theorems~\ref{thm:driftfbsde} and \ref{thm:weight}) the measure $\P_i$ is only 
    employed to compute the instantaneous conditional expectation; 
    thus, there is no  
    sample path matching required when taking $\tau = t_{i+1}$ and $t = t_i$ to 
    obtain $V^\mu(t_i,x)$, and when taking $\tau = t_i$ and $t = t_{i-1}$ to obtain $V^\mu(t_{i-1},x)$.

It can be observed in Fig.~\ref{fig:parallelvsweighted}(b) that 
some edges are multiply 
    represented in the distribution. 
If the drift term $K_i$ were restricted to be a deterministic function of $X_i$ (as is the case in \cite{exarchos2018stochastic,Exarchos2018,Exarchos2018a}), such a construction would represent 
    an unfaithful characterization of the path distribution because samples of the Brownian 
    process are independent, and thus should be sampled as in 
    Fig.~\ref{fig:parallelvsweighted}(c).
However, since $K_i$ itself is permitted to have a distribution, the overlapping of paths is justified as the drift having been selected so as to concentrate the paths in a certain part of the state space. 
Figure~\ref{fig:xk_joint_dist} illustrates why parallel sampling is naturally suited for representing 
    deterministic functions and why branch sampling is necessary for representing nontrivial joint 
    distributions ${(X_i,K_i)}$. 
    %
    
For any arbitrary function $G_i$ evaluated on path $\xi_i$, we have the almost sure convergence
\begin{align}
    \E_{\fP_i}[G_i(\xi_i)] = \sum_{j=1}^M \frac{1}{M} G_i(\xi^j_i) \rightarrow 
    \E_{\wt{\P}_i}[G_i(\xi_i)] \text{,} \label{eq:convexp}
\end{align}
    as the number of particles $M \rightarrow \infty$, where $\wt{\P}_i$ is the ideal discrete-time approximation of the distribution 
    $\P_i \circ \xi_i^{-1}$ under the
    Euler-Maruyama scheme \cite[Section~4.1.2]{del2013mean}.
It follows from the change of variables theorem \cite[Chapter~3, Theorem~3.6.1]{bogachev2007measure}
    that this expectation is exact up to the error due to time discretization
\begin{align*}
    \E_{\wt{\P}_i}[G_i(\xi_i)] 
    &\approx \int G_i(\xi_i) \, \P_i \circ \xi_i^{-1}(\mathrm{d} \xi_i) \\
    &= \int G_i(\xi_i(\omega)) \, \P_i (\mathrm{d} \omega) =: \E_{\P_i}[G_i(\xi_i)] \text{.}
\end{align*}
According to \cite[Chapter 10, Theorem 10.2.2]{kloeden2013numerical},
    when a linear growth condition in $x$ is imposed on $f^\mu_s$, $\sigma_s$, and $\ell^\mu_s$
    along with a few other conditions, the Euler-Maruyama scheme's error varies as  
    $\mathcal{O}((\Delta t)^{\sfrac{1}{2}})$.
When $\sigma_s$ is constant with respect to $x$, 
    the error bound improves to 
    $\mathcal{O}(\Delta t)$
    \cite[Chapter 10, Theorem 10.3.5]{kloeden2013numerical}.
Thus, our approximation converges with large numbers of particles and decreasing time intervals.

\begin{figure}
    \centering
        \includegraphics[width=0.97\linewidth]{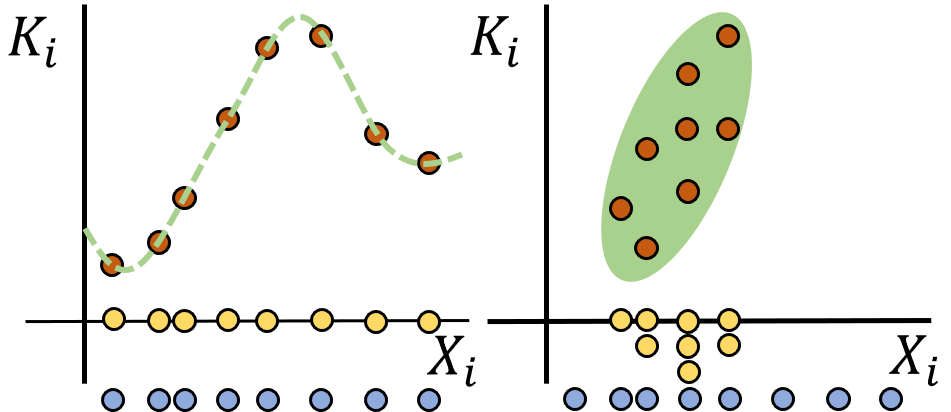}
        \caption{Comparison of the approximations of the joint distribution $\P_{(X_i,K_i)}$.
            \textbf{Left:} A parallel-sampling method is used to approximate a joint distribution
        wherein $K_i$ is a deterministic function of $X_i$.
        \textbf{Right:} A branch-sampling method is used to approximate a nontrivial joint
        distribution.
        The blue dots are the approximation of the distribution $\P_{X_i}$,
        the green curves are the ideal continuous distribution $\P_{(X_i,K_i)}$, 
        the red dots are its
        sampled approximation, and the yellow dots represent how many children that node
        in the tree will have.
        }
    \label{fig:xk_joint_dist}
\end{figure}

\subsection{Path Integral Least Squares Monte Carlo} \label{sec:pathlsmc}

To approximate the measure $\pR_{i+1}$ in Theorem~\ref{thm:weight} we use a 
    path integral-weighted measure 
\begin{align}
    \bR_{i+1} &:=
    \frac{1}{\eta} \sum_{j=1}^M \Theta^{j}_{i+1} \delta_{\xi_{i+1}^j} \text{,}
\end{align}
    where the weights for each heuristic value $\rho^j_{i+1}$ are 
\begin{align}
    \Theta^{j}_{i+1} = \exp(-\frac{1}{\lambda} \rho^j_{i+1}) \text{,}
    \label{eq:weighapprox}
\end{align}
    and $\eta$ is a normalizing constant.
The heuristic value is calculated as  
    \mbox{$\rho_{i+1}^j = \rho_{i+1}(\xi^j_{i+1})$}, taking care to exclude $w^j_{i,i+1}$
    so that its distribution remains Brownian.

In each step of the backward pass, we use $\bR_{i+1}$ and the
    value function approximation $V(x; \alpha_{i+1}) \approx V^\mu(t_{i+1}, x)$,
    parameterized by $\alpha_{i+1} \in \mathcal{A}$,
    where $\mathcal{A}$ is the parameter space, 
    to estimate the value function at the previous time step $V(x; \alpha_{i}) \approx V^\mu(t_{i}, x)$,
    by producing $\alpha_i \in \mathcal{A}$.
We assume that the parameterizarion $V(x; \alpha)$ results in a function that is $C^{2}$ for all $\alpha \in \mathcal{A}$, and
 approximate the optimization in \eqref{eq:thmcontinfk4} as
\begin{align}
    \alpha_i^* &= \argmin_{\alpha_i \in \mathcal{A}} 
    \E_{\bR_{i+1}}[(\wh{Y}_{i,i+1} - V(X_i;\alpha_i))^2] \nonumber \\
    &= \argmin_{\alpha_i \in \mathcal{A}} 
    \E_{\fP_{i+1}}[\Theta^j_{i+1} (\wh{Y}_{i,i+1} - V(X_i;\alpha_i))^2] \nonumber \\
    &= \argmin_{\alpha_i \in \mathcal{A}} 
    \frac{1}{\eta} \sum_{k=1}^M \Theta^{j}_{i+1} (\wh{y}^j_{i,i+1} - V(x^j_i;\alpha_i))^2 \text{,}
    \label{eq:alphasolve}
\end{align}
    where 
    \eqref{eq:bsdediff2} is approximated as
\begin{align}
    \wh{y}^j_{i,i+1} &\approx V(x_{i+1}^j;\alpha_{i+1}) + 
    (\ell^{\mu j}_i + z^{j \top}_{i+1} d^j_i) \Delta t  \text{.} 
    \label{eq:bsdediffapprox}
\end{align}

The novelty of this method over classic LSMC \cite{Longstaff2001}, developed for
    parallel-sampled paths, comes from (a) the observation that introducing the drift
    process $K_i$ with a non-trivial joint distribution $\P_{(X_i,K_i)}$
    validates the choice of branch-sampled path distributions;
    (b) we can weigh regression points using a heuristic that acts on the entire path,
    not just the immediate states; and
    (c) weighing as in \eqref{eq:weighapprox} has a particular interpretation as the selection
    of a measure with desirable properties for robustness using \eqref{eq:MinRelativeEntropyMeasure}.

\section{Forward-Backward RRT-SDE} \label{sec:fbrrt}

In this section, we present a novel algorithm (FBRRT-SDE) that uses rapidly exploring random trees to construct the graph $\mathcal{G}$ of samples for solving the corresponding system of FBSDEs.
The \mbox{FBRRT-SDE} algorithm is a particular numerical application of the generalized theory presented in Section \ref{sec:branchpathlsmc}.
The ultimate goal of the FBRRT-SDE algorithm is to produce the set of parameters $\{ \alpha_i \}_{i=1}^N$ which approximate the optimal value function as $V^*(t_i, x) \approx V(x;\alpha_i)$.
This is achieved by
generating a forward pass to produce a graph representation
$\mathcal{G}$
of the path measures $\{ \fP_i \}_{i=1}^N$.
Given that the optimal policy has the form~\eqref{eq:optpoli},
    we define the target policy
\begin{align}
    &\mu_i(x;\alpha_{i+1}) \label{eq:optpolicy2} \\
    &\quad = \argmin_{u \in U} \{ \ell(t_i,x,u) +  f(t_i,x,u)^\top \partial_x V(x;\alpha_{i+1}) \}
    \text{,} \nonumber
\end{align}
    so that it coincides with the optimal control policy
    when the value function approximation is exact.
The backward pass uses $\mathcal{G}$, $\mu_i$, and $\rho_{i+1}$ to produce $\alpha_i$, backwards
    in time.
    At each iteration $k$ of the algorithm, 
the policy cost $\ol{J}_k := \E[S_0[x_0, \mu]]$ 
associated with a set of parameterized policies is evaluated
    by sampling a parallel-sampled set of trajectories and computing the mean cost.
At the end of each iteration,
    nodes with high heuristic value $\rho_{i+1}$ are pruned from the tree $\mathcal{G}$,
    and new nodes are added in the forward pass in the next iteration.
This outer loop of the FBRRT-SDE algorithm is summarized in Algorithm~\ref{alg:fbsde}.

\begin{algorithm}
\caption{Forward-Backward RRT-SDE}\label{alg:fbsde}
\begin{algorithmic}[1]
    \Procedure{FBRRT-SDE}{$x_0$}
        \State $\wt{\mathcal{G}}$.init($\xi_0$)
        \For{$k = 1,\cdots,N_\text{iter}$} 
            \State $\mathcal{G} \gets$ 
                \Call{ForwardPass}{$\wt{\mathcal{G}},(\alpha_i)_i$}
            \Statex \Comment{Generate tree which represents $\{\fP_i\}_i$}
            \State $(\alpha_i)_i \gets$ \Call{BackwardPass}{$\mathcal{G}$} 
            \Statex \Comment{Approximate value functions $\{V(\cdot;\alpha_i)\}_i$}
            \State $\ol{J}_{k} \gets$ \Call{PolicyCost}{$x_0,(\alpha_i)_i$}
            \Statex \Comment{Evaluate computed policy 
            $\{\mu_i(\cdot; \alpha_{i+1})\}_i$}
            \State $\wt{\mathcal{G}} \gets$ \Call{Erode}{$\mathcal{G},(\alpha_i)_i$}
            \Statex \Comment{Prune tree to remove suboptimal paths}
        \EndFor
        \State \textbf{return} $(\alpha_i)_i$
    \EndProcedure
\end{algorithmic}
\end{algorithm}

\subsection{Kinodynamic RRT Forward Sampling} \label{sec:kinorrt}

We desire sampling methods that seek to explore the whole state space,
thus
    increasing the likelihood of sampling in the proximity of optimal trajectories.
For this reason, we chose a method inspired by kinodynamic RRT~\cite{lavalle2001randomized}.
The selection procedure for this method ensures that the distribution of the chosen
    particles is more uniformly distributed in a user-supplied region of interest,
 is
    more likely to select particles
    which explore the empty space, and is therefore less likely to oversample dense
    clusters of particles.

With some probability ${\eps^{\text{rrt}}_i \in [0,1]}$ we choose the RRT
    sampling procedure, 
    but otherwise we choose a particle uniformly from $\{x_i^{j}\}_{j=1}^M$, 
    each particle having equal weight.
This ensures that dense particle clusters will still receive more attention.
The choice of the parameter $\eps^{\text{rrt}}_i$ balances exploring the state space
    against refining the area around the current distribution.

For the drift values, that is, those sampled from the distribution $\kappa$ left unspecified in
    Section~\ref{sec:SDEBranched},
    we again choose a random combination of exploration and exploitation.
For exploitation we choose
\begin{align}
    K_i &= f(t_i, X_i, \mu_i(X_i;\alpha_i)) \text{,}
    \label{eq:koptdrift} 
\end{align}
and for exploration we choose
\begin{align}
    K_i &= f(t_i, X_i, u^\text{rand}) \text{,}
\end{align}
    where the control is sampled randomly from a user-supplied set
    $u^\text{rand} \sim U^\text{rand}$.
For example, for minimum fuel ($L_1$) problems where the control is bounded as $u \in [-1, 1]$ and
    the running cost is $L = |u|$, we select $U^\text{rand} = \{ -1, 0, 1 \}$
    because the policy \eqref{eq:optpolicy2} 
    is guaranteed to only return values in this discrete set.

Algorithm~\ref{alg:rrtbranchsamp} summarizes the implementation of 
    the RRT-based sampling procedure that produces the forward sampling tree $\mathcal{G}$.
The algorithm takes as input any tree with width $\wt{M}$ 
    and adds nodes at each depth until the
    width is $M$, the desired tree width.
In the first iteration there are no value function estimate parameters available
    to exploit,
    so we set $\eps^\text{rrt}=1$ to maximize exploration using RRT sampling.


\begin{algorithm}
\caption{RRT Branched-Sampling}\label{alg:rrtbranchsamp}
\begin{algorithmic}[1]
    \Procedure{ForwardPass}{$\mathcal{G},(\alpha_1,\ldots,\alpha_N)$}
    \For{$k = \wt{M}+1,\cdots,M $ } \Comment{Add node each loop}
        \For{$i = 0,\cdots,N-1 $ } \Comment{For each time step}
            \State $\{x^j_i\}_j \gets \mathcal{G}\text{.nodesAtTime}(i)$
            \If{$\eps^\text{rrt} > \kappa^\text{rrt} \sim \text{Uniform}([0,1])$}
                \State $x_i^{\text{rand}} \sim \text{Uniform}(\mathcal{X}^\text{roi})$
                \State $(x_i^{\text{near}},j^{\text{near}}) \gets 
                \text{Nearest}(\{x^j_i\}_j, x_i^{\text{rand}})$ \label{line:nearest}
            \Else
                \State $(x_i^{\text{near}},j^{\text{near}}) \sim \text{Uniform}(\{x^j_i\}_j)$
            \EndIf \Comment{$j^{\text{near}}$ is index of selected node}
            \If{$\eps^\text{opt} > \kappa^\text{opt} \sim \text{Uniform}([0,1])$}
                \State $u_i \gets \mu_i(x_i^{\text{near}};\alpha_{i+1})$ 
                \Comment{\eqref{eq:optpolicy2}}
            \Else
                \State $u_i \sim U^\text{rand}$
            \EndIf
            \State $k_i \gets f(t_i, x_i^\text{near}, u_i)$
            \State $w_i \sim \mathcal{N}(0,\Delta t I_{n})$
            \State $x_{i+1}^\text{next} \gets x_i^{\text{near}} + 
            k_i \Delta t + \sigma(t_i,x_i^{\text{near}}) w_i$ 
            \State $j^\text{next} \gets 
            \mathcal{G}\text{.addEdge}(i,j^\text{near}, (x^{\text{near}}_{i}, 
                k_i, x_{i+1}^\text{next}))$
            \State $\overrightarrow{\ell}_{0:i-1} 
                \gets \mathcal{G}\text{.getRunCost}(i-1, j^{\text{near}})$ 
            \State $\overrightarrow{\ell}_{0:i} 
                \gets \overrightarrow{\ell}_{0:i-1} 
                + \ell_i(x_i^{\text{near}}, u_i) \Delta t$ \label{line:pathint}
            \State $\mathcal{G}\text{.setRunCost}(i, j^\text{next}, 
                \overrightarrow{\ell}_{0:i})$
        \EndFor
    \EndFor
    \State \textbf{return} $\mathcal{G}$
    \EndProcedure
\end{algorithmic}
\end{algorithm}

\subsection{Path-Integral Dynamic Programming Heuristic} \label{sec:locentbw}

Next, we propose a heuristic design choice for the backward pass weighting variables
    $\rho_{i+1}$, and justify this choice with theoretical analysis.
A good heuristic will give large weights to paths likely to have low values over the whole
    interval $[0,T]$.
Thus, in the middle of the interval we care both about the current running cost and the expected cost.
A dynamic programming principle result
    following directly from 
    \cite[Chapter~4, Corollary~7.2]{fleming2006controlled}
    indicates that
\begin{align*}
    &V^*(0,x_0) =  \\
&\quad \min_{u[0,t_{i+1}]} E_{\P_{i+1}^{u}}[\int_{0}^{t_{i+1}} \ell(s,X_s,u_s) \, \ds 
    + V^*(t_{i+1},X_{i+1})] \text{,}
\end{align*}
    where $u_{[0,t_{i+1}]}$ is any control process in $U$ on the interval $[0,t_{i+1}]$ and
    $\P_{i+1}^{u}$ is the measure produced by the drift 
    $K_s = f(s,X_s,u_s)$.
Following this minimization, we choose the heuristic to be the discrete approximation of
\begin{align}
    \rho_{i+1} = \int_{0}^{t_{i+1}} \ell(s,X_s,u_s) \, \ds 
    + V^*(t_{i+1},X_{i+1}) \text{,}
    \label{eq:rhodef}
\end{align}
    where $u_{[0,t_{i+1}]}$ is chosen identically to how the control for the drift is produced.
    
Although the theory up to this point does not require $K_s$ to be a feasible drift under the 
    dynamic constraints of the SOC problem,
    for the proposed FBRRT-SDE algorithm the drift is always chosen as $K_s = f(s,X_s,u)$ for some randomly
    selected $u \in U$. 
In practice, the running cost is approximated by Euler-Maruyama, and the optimal value function
    is approximated by the latest estimate of the value function $V(x;\alpha_{i+1})$.
The running cost is computed in the forward sampling
    in line~\ref{line:pathint} of Algorithm~\ref{alg:rrtbranchsamp}.

\begin{algorithm}
\caption{Local Entropy Weighted LSMC Backward Pass}\label{alg:bptt}
\begin{algorithmic}[1]
    \Procedure{BackwardPass}{$\mathcal{G}$}
    \State $\{\xi^j_N\}_j \gets \mathcal{G}\text{.pathsAtTime}(N)$
    \State $\{x^j_N\}_j \gets \{\xi^j_N\}_j$
    \State $y_N \gets [g(x^1_N) \; \cdots \; g(x^M_N)]^\top$
    \State $\alpha_N \gets 
        \argmin_\alpha \sum_j \Theta_{N} (\widehat{y}^j_{N} - \Phi(x^j_{N}) \alpha)^2$ 
    \For{$i = N-1,\cdots,1$} \Comment{For each time step}
        \State $\{\xi^j_{i+1} \}_j 
        \gets \mathcal{G}\text{.pathsAtTime}(i+1)$
        \For{$j = 1,\cdots,M$} 
            \Comment{For each path}
            \State $(x^j_{i}, k^j_{i}, x^j_{i+1})  
                \gets \xi^j_{i+1}$
            \Comment{$x^j_{i} = x^j_{i,i+1}$, etc.}
            \State $y^j_{i+1} 
                \gets \Phi(x^{j}_{i+1}) \alpha_{i+1}$ 
            \Comment{\eqref{eq:thmcontinfk3}}
            \State $z^j_{i+1} 
                \gets \sigma^\top_{i+1}(x^j_{i+1}) 
                \partial_x \Phi(x^{j}_{i+1}) \alpha_{i+1}$ 
            \Comment{\eqref{eq:zv}}
            \State $\mu^j_{i} 
            \gets \mu_{i}(x^j_{i}; \alpha_{i+1})$
            \Comment{\eqref{eq:optpolicy2}}
            \State $d^j_{i} 
            \gets \sigma^{-1}_{i+1}(x^j_{i+1}) (f^\mu_i - k_i^j)$ 
            \State $\wh{y}^j_i \gets 
            y^j_{i+1} + (\ell^\mu_i 
            + z^{j \top}_{i+1} d^j_{i})\Delta t$ 
            \Comment{\eqref{eq:bsdediff2}}
            \State $\overrightarrow{\ell}_{0:i} 
                \gets \mathcal{G}\text{.getRunCost}(i, j)$ 
            \State $\rho^j_{i+1} \gets 
            y^j_{i+1} + \overrightarrow{\ell}_{0:i}$ 
            \Comment{\eqref{eq:rhodef}}
        \EndFor
        \State $\rho_{i+1} \gets \rho_{i+1} 
            - \min_j \{ \rho^j_{i+1} \}$ \label{st:expcond}
        \Comment{$\exp$ conditioning}
        \State $\Theta_{i+1} \gets \exp(-\sfrac{1}{\lambda}\rho_{i+1})$
        \Comment{\eqref{eq:changemeas}}
        \State $\alpha_{i} \gets 
        \argmin_\alpha \sum_j \Theta^j_{i+1} 
        (\widehat{y}^j_{i} - \Phi(x^j_i) \alpha)^2$ 
        \Comment{\eqref{eq:thmcontinfk4}} \label{line:lsmc}
    \EndFor
    \State \textbf{return} $(\alpha_1,\ldots,\alpha_N)$
    \EndProcedure
\end{algorithmic}
\end{algorithm}

Algorithm~\ref{alg:bptt} details the implementation of the backward pass with
    local entropy weighting.
Line~\ref{st:expcond} does not, theoretically, have an effect on the optimization,
    since it will come out of the exponential as a constant multiplier,
    but it has the potential to improve the numerical conditioning of the 
    exponential function computation
    as discussed in \cite[Chapter~5, equation (6.33)]{Goodfellow-et-al-2016}.
The $\lambda$ value is, in general, a parameter which must be selected by the user.
For some problems we choose to search over a series of possible $\lambda$
    parameters, evaluating each one with a backward pass and using the
    one that produces the smallest expected cost over a batch of trajectory rollouts
    executing the computed policy.

\subsection{Path Integral Erode} \label{sec:patherode}

After the backward pass of the algorithm we obtain updated approximations of the value
    function $\{V(\cdot;\alpha_i)\}_i$ 
    along with the tree $\mathcal{G}$ that represents the forward sampling
    path measures $\{\fP_i\}_i$.
To improve our approximation, we can use our value function estimates to create
    a new tree $\mathcal{G}'$ with
    new forward sampling measures $\{\fP'_i\}_i$ via the heuristic $\rho_i$ in \eqref{eq:rhodef}.

We have found experimentally that sampling a new tree from scratch is both wasteful
    and shows signs of catastrophic forgetting.
That is, the subsequent backward pass performs worse, since it has lost data samples
    which were important to form good function estimates.
On the other hand, simply adding more samples to the current tree can prove to be 
    unsustainable in the long run.
To keep the time complexity constant between iterations, we propose to bound 
    the number of samples at each time step.
After each
    backward pass we remove as many samples as were added in the forward pass,
    ``eroding'' the tree before the forward pass ``expands'' it.


The algorithm starts with a tree of width $M$ and ends with a tree
    of width $\wt{M}$ at every depth.
We begin at the end of the trajectory $i = N$ and remove the nodes $\{x_N^j\}_{j=1}^M$ 
    with highest $\rho^j_{N}$ value until there are only $\wt{M}$ nodes 
    left at depth $N$.
We proceed in a similar fashion backwards down the tree, removing nodes with high $\rho^j_i$
    values.
However, due to the tree structure of the path measures,
    if we remove nodes that have children we disconnect the paths and ruin the assumed
    structure.
Thus, we only remove nodes that have no children.
The implementation of this algorithm is detailed in Algorithm~\ref{alg:eroderrt}.


\begin{algorithm}
\caption{Path Integral Erode}\label{alg:eroderrt}
\begin{algorithmic}[1]
    \Procedure{Erode}{$\mathcal{G},(\alpha_i)_i$}
    \For{$i = N,\cdots,1$} \Comment{For each time step}
        \State $\{\rho^j_{i}\}_j \gets 
        \mathcal{G}\text{.getHeuristics}(t_{i})$
        \ForAll{$j' \in \text{sortDescending}(\{\rho^j_{i}\}_j)$ }
            \If{$\mathcal{G}\text{.hasNoChildren}(x^{j'}_{i})$}
                \State $\mathcal{G}\text{.removeParentEdge}(x^{j'}_{i})$
                \State $\mathcal{G}\text{.removeNode}(x^{j'}_{i})$
            \EndIf
            \If{$\mathcal{G}\text{.numNodes}(t_{i}) = \wt{M}$}
                \State \textbf{break}
            \EndIf
        \EndFor
    \EndFor
    \State \textbf{return} $\mathcal{G}$
    \EndProcedure
\end{algorithmic}
\end{algorithm}

\subsection{Function Approximation}

In our implementation of the FBRRT-SDE algorithm,
    the value function is represented by 2nd order multivariate Chebyshev polynomials.
Specifically, we use all products of the basis functions 
    $\bigcup_{j=1}^n \{1, x_j, 2x_j^2-1\}$
    with polynomial degree 2 or lower, namely,
\begin{align*}
    \Phi(x) := (1, x_1, \ldots, x_n, 2x_1^2-1, \ldots, 2x_n^2-1, \\
    x_1 x_2, \ldots, x_1 x_n, x_2 x_3, \ldots, x_2, x_n, \cdots, x_{n-1} x_n) \text{.}
\end{align*}
For better conditioning, points are first normalized to the interval $[-1,1]^n$
    based on a parameterized region of interest to obtain the basis functions
    ${\Phi(\ldots,(x_j-a^\text{offset}_j)/a^\text{scale}_j,\ldots)}$.

\subsection{Computational Complexity}

The computational complexity of the forward pass is dominated by the nearest neighbors
    search.
In our implementation we use brute force search so the complexity of
    line~\ref{line:nearest} in Algorithm~\ref{alg:rrtbranchsamp} is $\mathcal{O}(M)$,
    but the complexity
   can be lowered to $\mathcal{O}(\log(M))$  if a kd-tree is used for nearest neighbor search, instead.
Thus, the forward pass complexity is $\mathcal{O}(N M^2)$ (or $\mathcal{O}(N M \log(M))$ with a kd-tree).

The backward pass complexity is dominated by the least squares regression at line~\ref{line:lsmc}
    in Algorithm~\ref{alg:bptt}.
A QR-decomposition solver is used for the regression, so if $k$ is the number of basis functions used in $\Phi(x)$, and
    $M \gg k$ then its complexity is $\mathcal{O}(M k^2)$.
Thus, the complexity of the backward pass is $\mathcal{O}(N M k^2)$.

\section{Numerical Examples} \label{sec:numresults}

We evaluated the FBRRT-SDE algorithm by applying it to three challenging nonlinear stochastic
    optimal control problems.
For all three problems, the terminal cost is taken to be a quadratic function centered at the origin
    and the running cost is taken to be an $L_1$/min-fuel cost,
    which makes these problems challenging for traditional solution methods.
For the first two problems the control input was restricted in the interval $U = [-1, 1]$.
For the third problem the control is restricted to $U = [-1, 1]^2$.
    
\subsection{$L_1$ Double Integrator}
    
\begin{figure*}
    \centering
    \subfloat[Parallel-sampled FBSDE \cite{Exarchos2018}]{
        \includegraphics[width=0.35\textwidth]{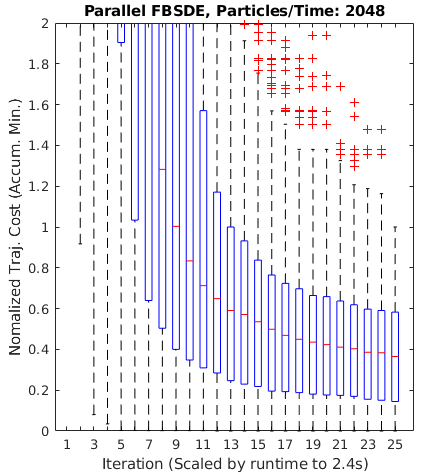}}
    \centering
    \subfloat[FBRRT-SDE]{
        \includegraphics[width=0.35\textwidth]{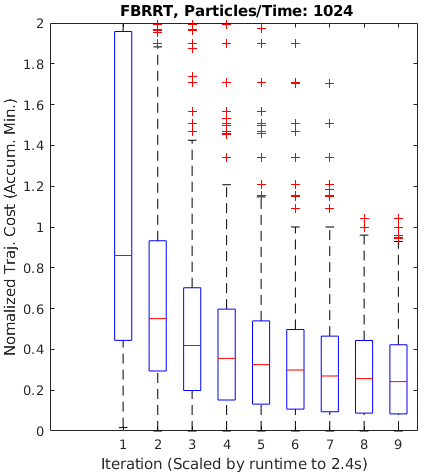}}
        \\
    \centering
    \subfloat[Trajectory samples for parallel-sampled FBSDE]{
        \includegraphics[width=0.37\linewidth]{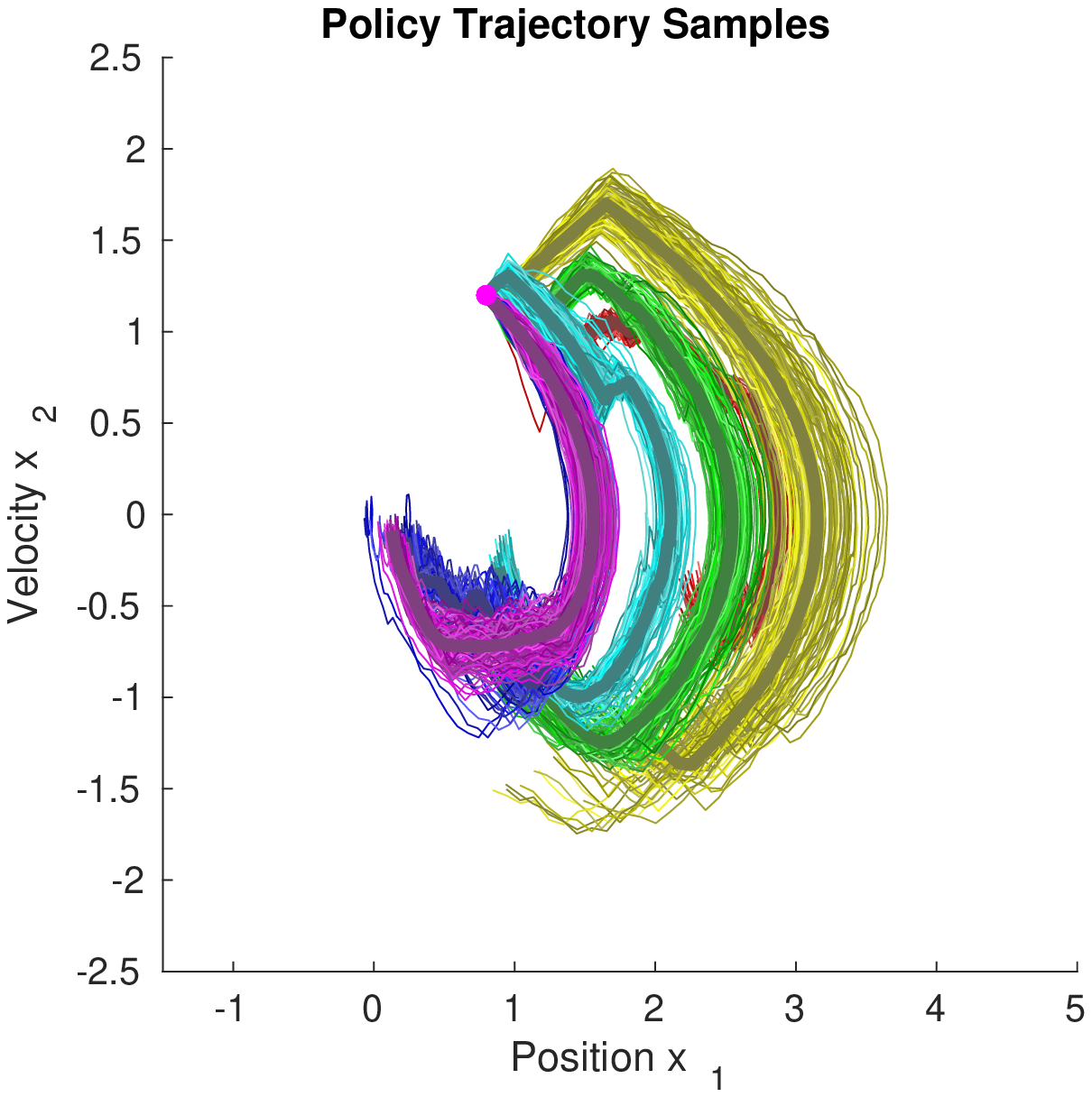}}
    \centering
    \subfloat[Trajectory samples for FBRRT-SDE]{
        \includegraphics[width=0.37\linewidth]{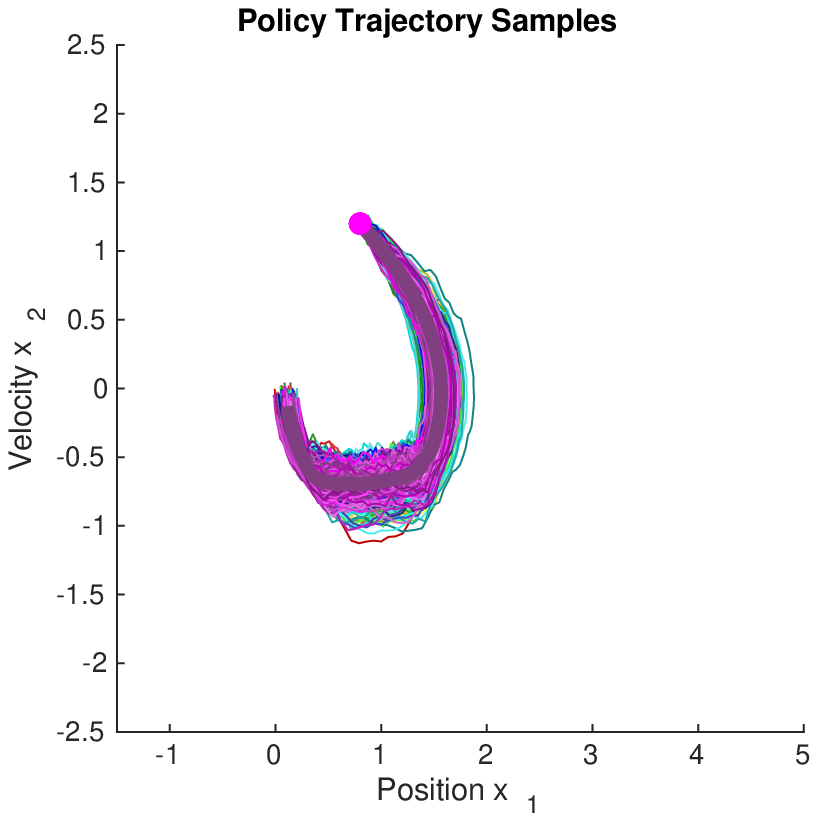}}
    \caption{\textbf{(a-b)} Comparison of parallel-sampled FBSDE \cite{Exarchos2018} and FBRRT-SDE 
         for the $L_1$ double integrator problem for
         random initial states. Expected trajectory costs for the computed policies
         are normalized across different initial conditions.
         The relative timings of iterations are visualized by their placement on the x-axis,
         were the right edge of each figure represents $2.4$s into runtime.
         \textbf{(c-d)} Trajectory samples from policies
         generated after the first 6 iterations.
         The first iteration is colored red, followed by yellow, green, cyan, dark blue, 
         and magenta. Thick lines are mean trajectories.}
\label{fig:dblintbatch}
\end{figure*}

In order to compare the proposed FBRRT-SDE algorithm to the  parallel sampled techniques in \cite{Exarchos2018}, which we denote below as parallel-sampled FBSDE,
we considered the double integrator system
with
\begin{align}
 \begin{bmatrix}
\dX_{s}^{(1)}\\
\dX_{s}^{(2)}
    \end{bmatrix} \nonumber
    &=
    \begin{bmatrix}
X_{s}^{(2)}\\
u
    \end{bmatrix}\ds+
    \begin{bmatrix}
0.01 & 0\\
0 & 0.1
        \end{bmatrix}
        \begin{bmatrix}
\dW_{s}^{(1)}\\
\dW_{s}^{(2)}
    \end{bmatrix} \text{,}
\end{align}
    with $L_1$ running cost, i.e.,
\begin{align}
    \inf_{u_{[0,T]}}  \E_{{\Q}}\bigg [\, \int_0^T c_0 |u_s| \ds 
        + \sum_{j=1}^n c_j (X_T^{(j)})^2 \,\bigg ] \text{,}
\end{align}    
    where $c_0, c_1, c_2$ are scalar parameters.
When the system starts with positive position and velocity, the optimal policy is
    to decelerate to a negative velocity, coast for a period of time so that fuel
    is not used, and then accelerate to reach the origin.
For this example, the number of particles per time step is $M=1,024$, the number of time steps is $N=64$,
    and the erode particle number is $\wt{M}=M/2$.



As shown in Figs.~\ref{fig:dblintbatch}(c)-(d), 
    the parallel-sampled FBSDE takes a significant number of iterations
    to begin converging to the near optimal policy, while the proposed method produces a
    near-optimal policy at the first iteration.
The algorithm converges to the optimal policy.



We also compared the convergence speed and robustness of the two methods by 
    randomly sampling different starting
    states and evaluated their relative performance over a number of trials.
For each of $30$ random initial states $x_0$ we ran $20$ trials of each method
    for a number of iterations, each iteration producing an expected cost for
    the computed policy.
We normalized the final costs across the initial states by dividing all costs for a particular
    initial state by the largest cost obtained across both methods.
For each iteration, we assign the value of the accumulated minimum value across previous
    iterations for that trial, i.e., the value is the current best cost after running
    that many iterations, regardless of the current cost.
We aggregated these values across initial states and trials into the box plots in
    Fig.~\ref{fig:dblintbatch}.
Since the FBRRT-SDE is significantly slower than the parallel-sampled FBSDE per iteration due to the nearest
    neighbors calculation,
    we scale each iteration by the runtime.
Note that every iteration of FBRRT-SDE after the first one requires approximately half the runtime,
    since only half of the eroded tree needs resampling.
In summary, the FBRRT-SDE converges faster using fewer iterations than the parallel-sampled FBSDE, 
    and does so with half as many particle samples.

\subsection{$L_1$ Double Inverted Pendulum}

\begin{figure}
    \centering
    \subfloat[Mean cost distribution for $x_0^\text{vert}$]{
        \includegraphics[width=0.97\linewidth]{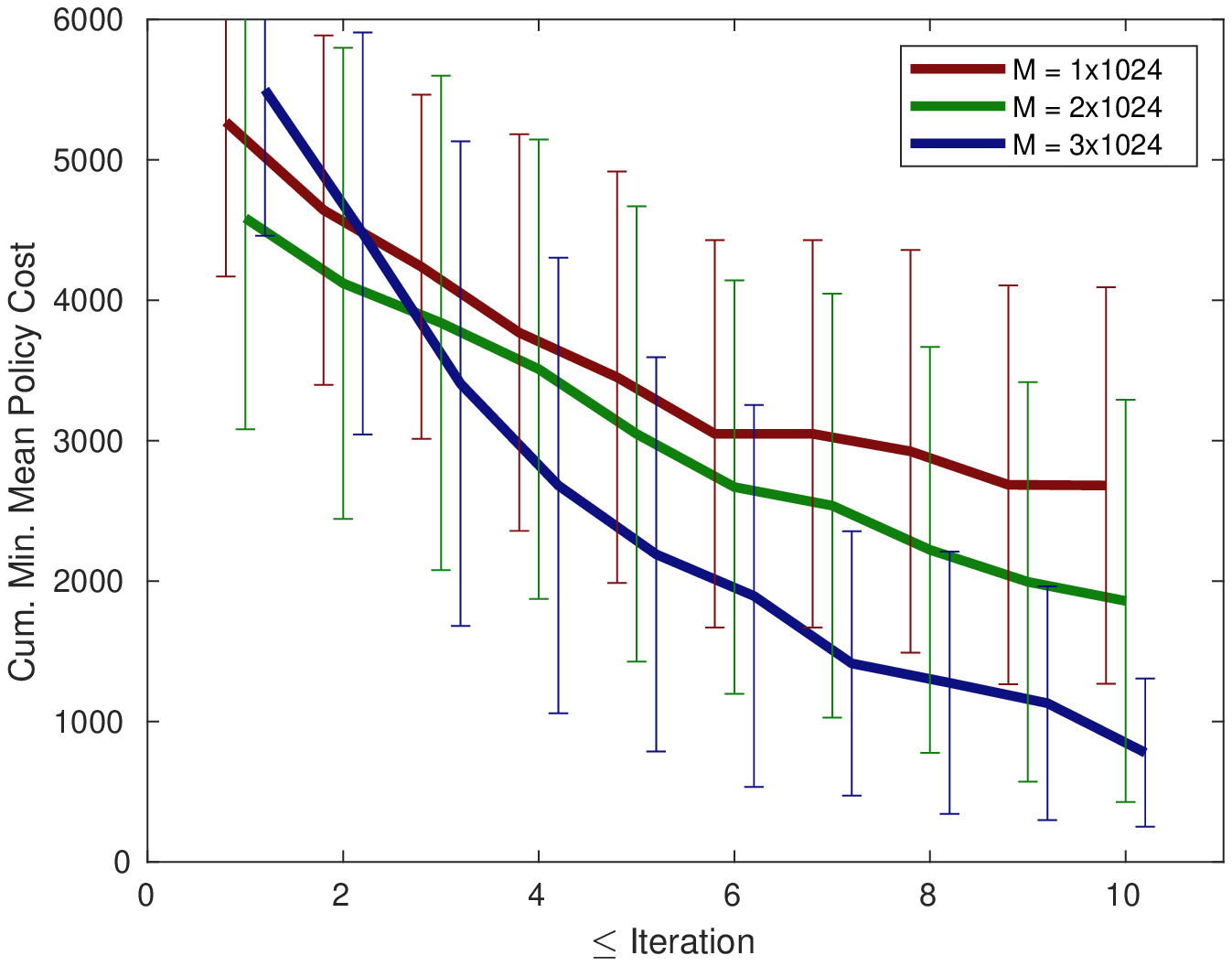}}
    \\
    \subfloat[Mean cost distribution for $x_0^\text{off}$]{
        \includegraphics[width=0.97\linewidth]{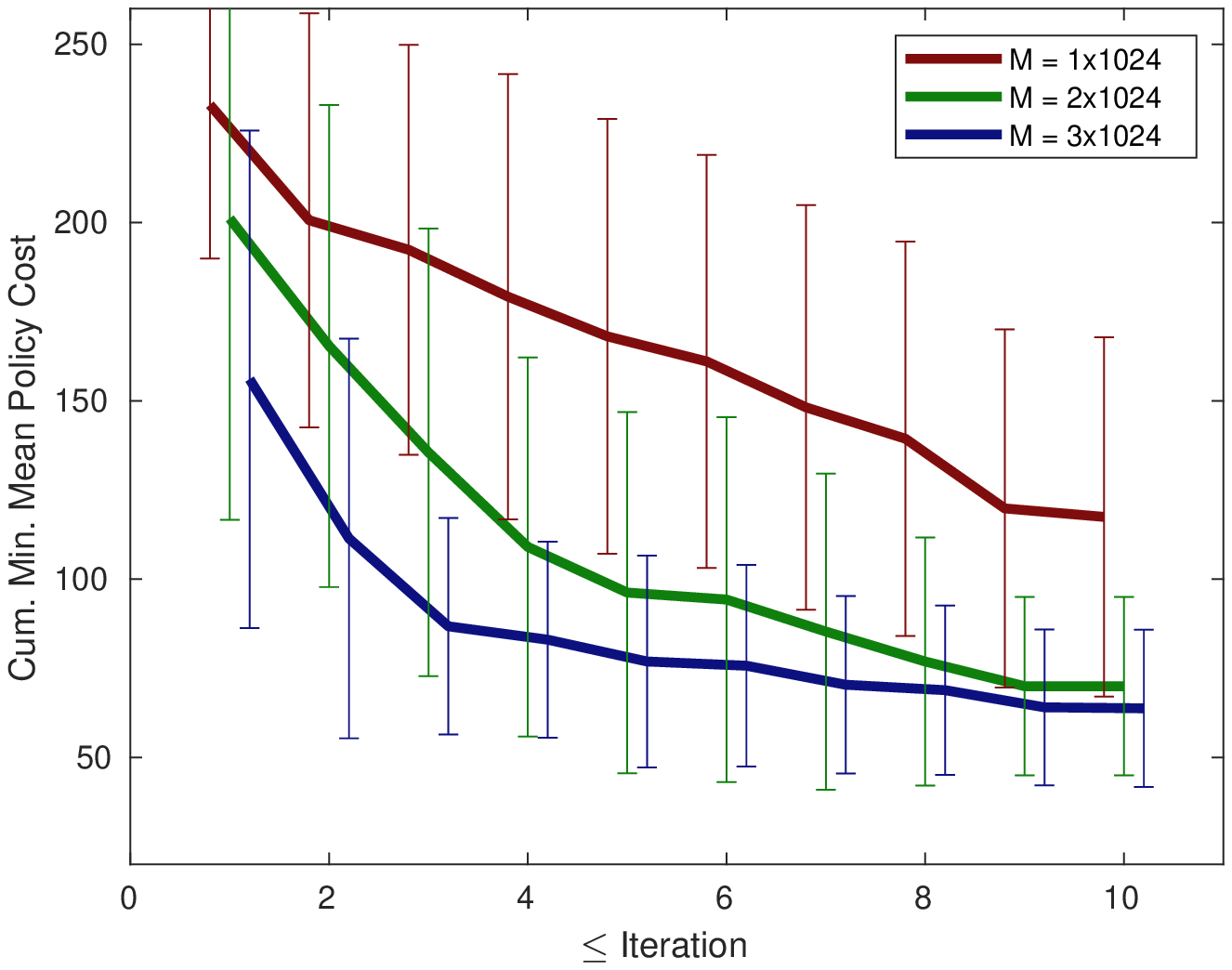}}
        \caption{Mean policy cost statistics for $L_1$ double inverted pendulum problem.
             The mean bars and standard deviation whiskers characterize the distribution over 30 trials, 
             where the value
             for each iteration is the accumulated minimum of the values over all previous
             iterations in that trial up to and including that iteration.
             $M$ particles are used per time step in each condition.
            }
        \label{fig:dpendstats}
\end{figure}
In order to study the proposed FBRRT-SDE algorithm on a highly nonlinear system in higher dimensions, consider the double inverted pendulum with state space dimension \mbox{$n=4$} 
presented in \cite{tedrake2009underactuated}, but with added damping friction to the joints.
Thus, the dynamics are in the form $\dX_s = f(X_s, u_s) \, \ds + \sigma \, \dW_s$, where
$f$ is given in \eqref{eq:DoubleInvertedPendulumF}, 
where $d_0 = 10,d_1 = 0.37,d_2 = 0.14,d_3 = 0.14, f_1 = 4.9,f_2 = 5.5,f_3 = 0.1, f_4 = 0.1$
    are scalar parameters of the system,
and where $\sigma = \mathrm{diag}[  0.03,0.03,0.18,0.18]$.
The associated optimal control problem is
\begin{align}
\hspace*{-2mm}
    \inf_{u_{[0,T]}}  \E_{{\Q}} \bigg [\, \int_0^T c_0 |u_s| \ds 
        + \sum_{j=1}^n c_j (X_T^{(j)})^2 \,\bigg ] \text{,}
\end{align}
    where $c_0,c_1,c_2,c_3,c_4$ are scalar parameters.
With our unoptimized
implementation (where brute force search is used for nearest neighbors)
    this example takes approximately $10$~sec to complete
    the first iteration and $80$~sec to complete 10 iterations for $M = 1 \times 1024$
    ($11$~sec and $85$~sec respectively for $M = 3 \times 1024$).
\begin{figure*}
\begin{multline}
\label{eq:DoubleInvertedPendulumF}
f\left(x,u\right)\equiv f\left(\left[\begin{array}{cccc}
\alpha & \beta & \omega & \psi\end{array}\right]^{\top},u\right)
\\
=\left[\begin{array}{c}
\omega\\
\psi\\
\frac{d_{3}\left(d_{2}\psi^{2}\sin\beta+2d_{2}\omega\psi\sin\beta-f_{3}\omega+f_{2}\sin(\alpha+\beta)-f_{1}\sin\alpha\right)+d_{2}\cos\beta\left(d_{2}\omega^{2}\sin\beta+f_{4}\psi-f_{2}\sin(\alpha+\beta)\right)+d_0 d_{3}\,u}{d_{1}d_{3}+2d_{2}d_{3}\cos\beta-d_{2}^{2}\cos^{2}\beta}\\
\frac{-\left(d_{1}+2d_{2}\cos\beta\right)\left(d_{2}\omega^{2}\sin\beta+f_{4}\psi-f_{2}\sin(\alpha+\beta)\right)-d_{2}\cos\beta\left(d_{2}\psi^{2}\sin\beta+2d_{2}\omega\psi\sin\beta-f_{3}\omega+f_{2}\sin(\alpha+\beta)-f_{1}\sin\alpha\right)-d_0 d_{2}\cos\beta\,u}{d_{1}d_{3}+2d_{2}d_{3}\cos\beta-d_{2}^{2}\cos^{2}\beta}
\end{array}\right] .
\end{multline}
\end{figure*}

Two initial conditions were evaluated, $x_0^\text{vert} = [0,0,0,0]^\top$, where the bars are vertically down and motionless, and \mbox{$x_0^\text{off} = [\pi/10,\pi/10,0,0]^\top$}, where the angles of both bars are slightly perturbed from $x_0^\text{vert}$ by $18^\circ$.
The number of time steps is taken to be $N=80$ and the erode particle number is selected as $\wt{M}=(3/4)M$. The evaluation of these conditions over 30 trials with differing numbers of particles $M$ is provided in \mbox{Fig.~\ref{fig:dpendstats}}.

Since the initial conditions of the two experiments are close, their optimal values
    should also be close.
Despite having comparable optimal values, the $x_0^\text{off}$ condition
    converges far more rapidly than the $x_0^\text{vert}$ condition.
Slightly perturbing the initial condition vastly improved the performance of the algorithm
    for this problem.
The reason the $x_0^\text{vert}$ condition performs poorly is likely because the system is
    very sensitive in that region and a localized policy results in a bifurcation of trajectory
    densities.
If the differing groups of trajectories have similar heuristic values,
    the value function approximation tries to fit a function to groups of particles in
    different sides of the state space, resulting in poor accuracy for either group.
When the $x_0^\text{off}$ condition is used, there is less ambiguity in which trajectory distributions
    are near-optimal, 
    resulting in better performance.

\subsection{$L_1$ Linearized Quadcopter}

\begin{figure}
    \centering
    \includegraphics[width=0.97\linewidth]{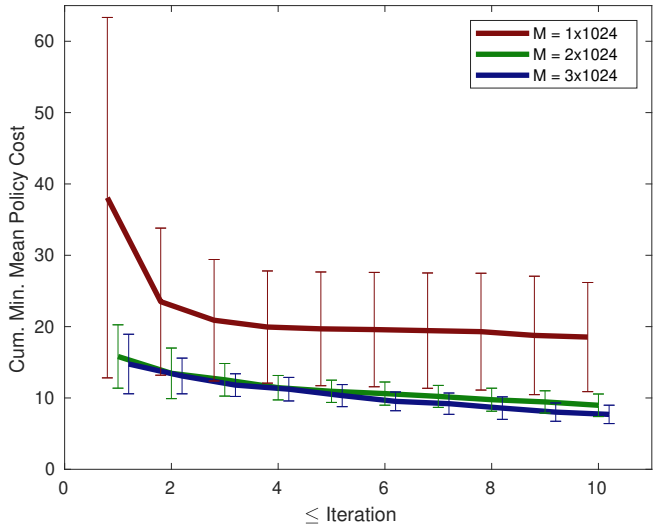}
        \caption{Mean policy cost statistics for $L_1$ linearized quadcopter problem.
             The mean bars and standard deviation whiskers characterize the distribution over 30 trials, 
             where the value
             for each iteration is the accumulated minimum of the values over all previous
             iterations in that trial up to and including that iteration.
             $M$ particles are used per time step in each condition.
            }
        \label{fig:quadstats}
\end{figure}

In order to demonstrate the proposed algorithm on a 
higher-dimensional system,
  we  considered the linearized dynamics of a quadcopter having state space dimension \mbox{$n=8$},
adapted from \cite{sabatino2015quadrotor}.
The dynamics are in the form $\dX_s = f(X_s, u_s) \, \ds + \sigma \, \dW_s$, where
\begin{align*}
f(x,u) &= f(
\begin{bmatrix}
\phi & \theta & p & q & u & v & x & y
\end{bmatrix}^\top,
\begin{bmatrix}
\tau_x & \tau_y
\end{bmatrix}^\top) \\
&=
\begin{bmatrix}
    p & q & d \tau_x & d \tau_y &
    -g \theta & g \phi & u & v
\end{bmatrix}^\top
\end{align*}
where $d = 4.1, g = 9.8$ are scalar parameters of the system,
and where $\sigma = \mathrm{diag}[10^{-5},10^{-5},0.2,0.2,0.002,0.002,10^{-5},10^{-5}]$.
The associated optimal control problem is
\begin{align}
\hspace*{-2mm}
    \inf_{u_{[0,T]}}  \E_{{\Q}} \bigg [\, \int_0^T |u^{(1)}_s| + |u^{(2)}_s| \ds 
        + \sum_{j=1}^n c_j (X_T^{(j)})^2 \,\bigg ] \text{,}
\end{align}
    where $c_1=c_2=c_3=c_4=c_5=c_6=1, c_7=c_8=100$ are scalar parameters.
With our unoptimized
implementation (where brute force search is used for nearest neighbors)
    this example takes approximately $5$~sec to complete
    the first iteration and $37$~sec to complete 10 iterations for $M = 1 \times 1024$
    ($18$~sec and $110$~sec respectively for $M = 3 \times 1024$).
    The results are shown in Fig.~\ref{fig:quadstats}.

For this example, most progress in terms of convergence occurs in the first iteration.
Adding more particles significantly improved progress both in the first iteration
    and after several iterations.
However, there were diminishing returns as the number of particles increased from $2 \times 1024$ to $3 \times 1024$,
    likely due to the fact that the first iteration is already near optimal.
Of interest is also the observation that the time required to solve  this 8-dimensional problem was not significantly longer
    than the previous 4-dimensional example.
This suggests that the time-complexity of the proposed method is more related to the complexity
    of the control problem itself rather 
    than the dimensionality of the state space.

\section{Conclusions and Future Work}

We have proposed a novel generalization of the FBSDE approach to solve
    stochastic optimal control problems, combining branched sampling techniques
    with weighted least squares function approximation to greatly expand the flexibility
    of these methods.
    By leveraging the efficient space-filling properties of RRT methods, we have demonstrated
    that our method significantly improves the convergence properties of previous FBSDE numerical methods.
We have shown how the proposed method works hand-in-hand with a
    local entropy-weighted LSMC method, concentrating function approximation in
    the regions where optimal trajectories are most likely to be dense.
We have demonstrated that FBRRT-SDE can generate feedback control policies for high-dimensional
    nonlinear stochastic optimal control problems.

Several of the design choices exposed by our approach offer significant
    opportunities for further research.
First, although in this paper we have employed the most basic of the RRT
    algorithms, there has been almost two decades of development in this field.
Employing more 
elaborate methods may improve the forward sampling even further.
In addition, in this paper we did not discuss state constraints or obstacles.
Since RRT methods are naturally designed to accommodate obstacles, the methods
    proposed here should be extendable to those problems as well.
For example, \cite{yong1999stochastic} reports a version of the Feynman-Kac
    theorem which involves FBSDEs with random stopping times and their
    relationship to boundary value PDEs.

Another area of research worth investigating is to find
other methods of
    value function representation.
In this paper, we use a rather simple parameterization of the value functions, 
    though this simplicity offers some distinct benefits.
    Specifically,
quadratic basis functions result in gradients that are linear, and policies that
    are typically stable.
In regions where particles are sparse, the convexity of the value function representation
    naturally drives the system back towards the particle distribution.
Further investigations might reveal other value function parameterizations 
with potentially better representation power than quadratic functions, 
while maintaining the benefits and the nice properties of quadratic basis functions.

\bibliographystyle{IEEEtran}
\bibliography{library}


\begin{IEEEbiography}[{\includegraphics[width=1in,clip,keepaspectratio]{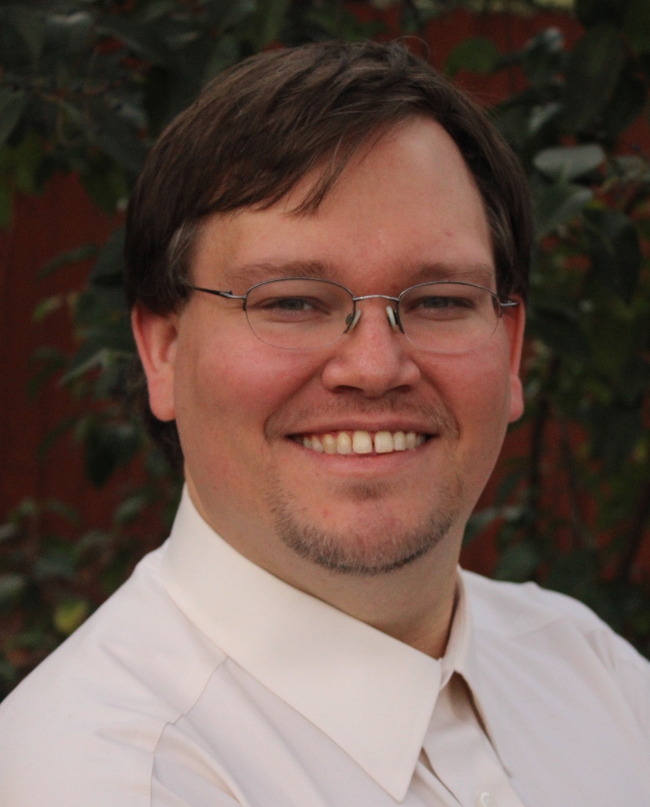}}]{Kelsey P. Hawkins}
received B.S. degrees in Applied Mathematics and Computer Science from North Carolina State University in 2010 and the M.S. degree in Computer Science
from the Georgia Institute of Technology in 2012. 
He is now a Ph.D. candidate in Robotics at the Georgia Institute of Technology in the Institute for Robotics and Intelligent Machines. 
His research interests include stochastic optimal control, differential games, and safe human-robot control design, with applications in autonomous driving, industrial robotics, and healthcare robotics.
\end{IEEEbiography}

\begin{IEEEbiography}[{\includegraphics[width=1in,height=1.25in,clip,keepaspectratio]{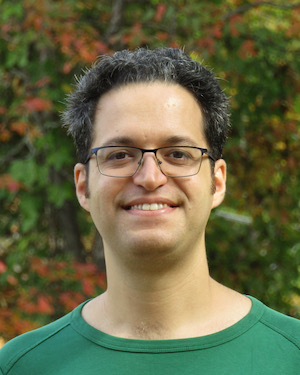}}]{Ali Pakniyat}(M'14) received the B.Sc. degree in Mechanical Engineering from Shiraz University in 2008, the M.Sc. degree in Mechanical Engineering 
from Sharif University of Technology in 2010, and the Ph.D. degree in Electrical Engineering from McGill University in 2016. After holding a Lecturer position at the Electrical and Computer Engineering department of McGill University, a postdoctoral research position in the department of Mechanical Engineering at the University of~Michigan, Ann Arbor, and a postdoctoral research position in the Institute for Robotics and Intelligent Machines at Georgia Institute of Technology, he joined the department of Mechanical Engineering at the University of Alabama in 2021 as an Assistant Professor. His research interests include deterministic and stochastic optimal control, nonlinear and hybrid systems, analytical mechanics and chaos, with applications in the automotive industry, robotics, sensors and actuators, and mathematical finance.
\end{IEEEbiography}

\begin{IEEEbiography}[{\includegraphics[width=1in,height=1.25in,clip,keepaspectratio]{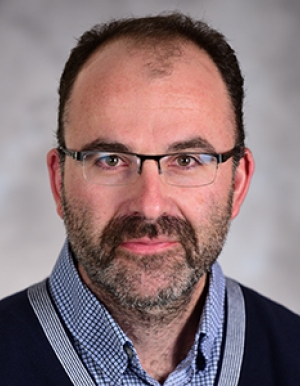}}]{Evangelos Theodorou} is an associate professor in the Daniel Guggenheim School of
Aerospace Engineering at the Georgia Institute
of Technology. He is also affiliated with the
Institute of Robotics and Intelligent Machines.
Theodorou earned a Bachelors and Masters degree in Electrical Engineering from the Technical
University of Crete, Greece, a Masters degree
in Computer Science and Engineering from the
University of Minnesota, and a a Ph.D. in Computer Science from the University of Southern
California. His research interests span the areas of stochastic optimal
control, robotics, machine learning, and computational neuroscience.
\end{IEEEbiography}

\begin{IEEEbiography}[{\includegraphics[width=1in,height=1.25in,clip,keepaspectratio]{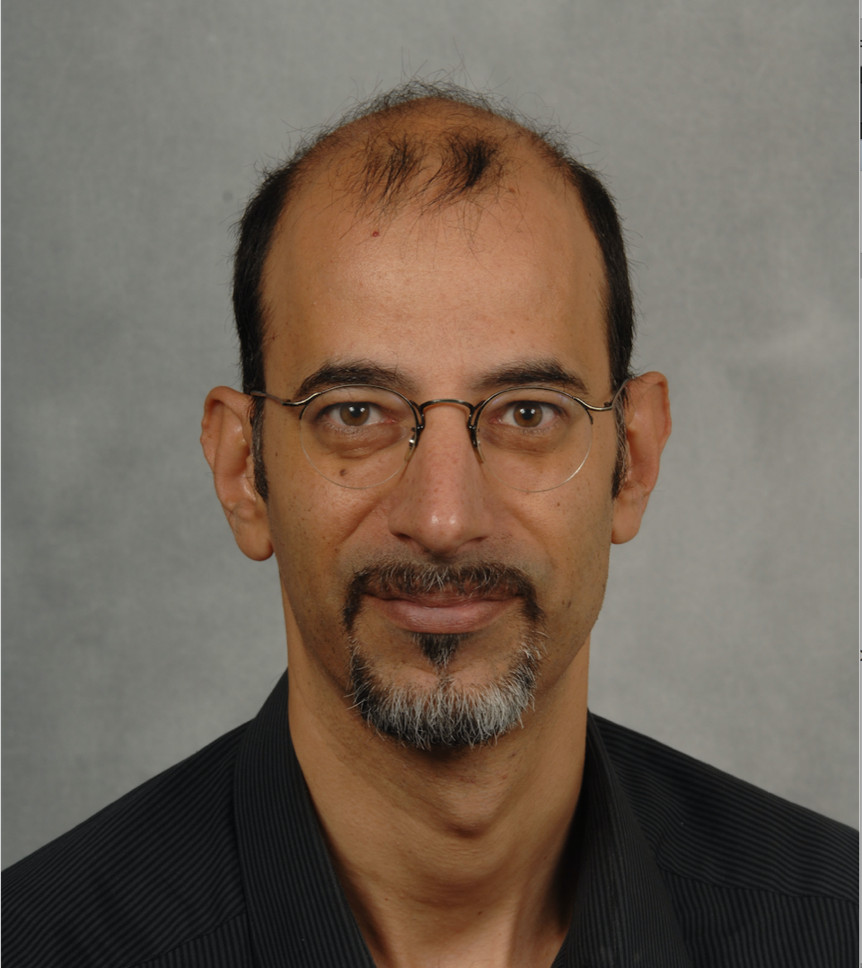}}]{Panagiotis Tsiotras} (F'19)
	is the David and Andrew Lewis Chair Professor in the Daniel Guggenheim School of Aerospace Engineering at the Georgia Institute of Technology (Georgia Tech).
	He holds degrees in Aerospace Engineering, Mechanical Engineering, and Mathematics.
	He has held visiting research appointments at MIT, JPL, INRIA Rocquencourt, and Mines ParisTech. 
	His research interests include optimal control of nonlinear systems and ground, aerial and space vehicle autonomy. He has served in the Editorial Boards of the Transactions on Automatic Control, the IEEE Control Systems Magazine, the AIAA Journal of Guidance, Control and Dynamics, the Dynamic Games and Applications, and Dynamics and Control. 
	He is the recipient of the NSF CAREER award, the Outstanding Aerospace Engineer award from Purdue, and the Technical Excellence Award in Aerospace Control from IEEE. 
	He is a Fellow of AIAA, IEEE, and AAS.
\end{IEEEbiography}

\end{document}